\documentclass[a4paper]{article}
\usepackage{amssymb,amsmath,amsthm}
\usepackage[lmargin=25mm,
rmargin=25mm, tmargin=25mm, bmargin=20mm]{geometry}
\usepackage{fancyhdr}
\newtheorem{thm}{Theorem}[section]
\newtheorem{prop}[thm]{Proposition}
\newtheorem{lem}[thm]{Lemma}
\newtheorem{cor}[thm]{Corollary}
\newenvironment{dfn}{\medskip\refstepcounter{thm}
\noindent{\bf Definition \thesection.\arabic{thm}.}}{\medskip}

\newenvironment{ex}{\medskip\noindent{\bf Example.}}{\medskip}

\newenvironment{remark}[1][Remark.]{\begin{trivlist}
\item[\hskip \labelsep {\bfseries #1}]}{\end{trivlist}}

\newenvironment{note}[1][Note]{\begin{trivlist}
\item[\hskip \labelsep {\bfseries #1}]}{\end{trivlist}}

\def\R{\mathbb{R}}
\def\C{\mathbb{C}}
\def\P{\mathbb{P}}
\def\O{\mathbb{O}}
\def\Z{\mathbb{Z}}
\def\H{\mathbb{H}}
\def\SS{\mathbb{S}}

\def\d{\mathrm{d}}

\DeclareMathOperator\id{id}

\DeclareMathOperator\SO{SO}
\DeclareMathOperator\SU{SU}

\DeclareMathOperator\GL{GL}
\DeclareMathOperator\GG{G}

\DeclareMathOperator\dett{det}

\DeclareMathOperator\Span{Span}
\DeclareMathOperator\Ker{Ker}
\def\Im{{\rm Im }}

\DeclareMathOperator\Spin{Spin}

\DeclareMathOperator\Vol{Vol}
\DeclareMathOperator\vol{vol}
\DeclareMathOperator\Div{div}
\DeclareMathOperator\Ree{Re}
\DeclareMathOperator\Imm{Im}
\DeclareMathOperator\Hess{Hess}
\newcommand{\Dslash}{\ensuremath \hspace{1pt}\raisebox{0.025cm}{\slash}\hspace{-0.24cm} D}

\begin{document}

%+Title
\title{Calibrated Submanifolds}
\author{Jason D. Lotay\\ {\normalsize University College London}}
\date{}
\maketitle

\begin{abstract}
We provide an introduction to the theory of calibrated submanifolds through the key examples related with special holonomy.  We focus on calibrated geometry in Calabi--Yau, $\GG_2$ and $\Spin(7)$ manifolds, and describe fundamental results and techniques in the field. 
\end{abstract}

\tableofcontents

\section*{Introduction}

A key aspect of mathematics  is the study of variational problems.  These can vary from the purely analytic to the very geometric.  A classic geometric 
example is the study of geodesics, which are critical points for the length functional on curves.  As we know, understanding the geodesics of a given 
Riemannian manifold allows us to understand some of the ambient geometry, for example the curvature.  
The higher dimensional analogue would be to study critical points for the 
volume functional, and we would hope (and it indeed turns out to be the case) that these critical points, called \emph{minimal submanifolds}, encode crucial aspects of 
the geometry of the manifold.  

Just like the geodesic equation, we would expect (and it is true) that minimal submanifolds are defined by a (nonlinear) second order partial differential equation.  
  Such equations are very difficult to solve in general, so a key idea is to find a special class of minimal submanifolds, called \emph{calibrated submanifolds}, which are instead defined by a first order partial 
  differential equation.  The definition of calibrated submanifolds is motivated by the properties of complex submanifolds 
in K\"ahler manifolds, and turns out to be useful in finding minimizers for the volume functional rather than just critical points. 
However, finding examples outside the classical complex setting turns out to be difficult, 
leading to important methods coming from a variety of sources, as well as motivating the
study of the deformation theory of these objects.

Calibrated submanifolds naturally arise when the ambient manifold 
 has \emph{special holonomy}, including holonomy $\GG_2$.  In this situation, we 
  would hope that the calibrated submanifolds encode even more, finer, information about the ambient manifold, potentially leading to the construction of  new 
  invariants.  In this setting, there is also a relationship between calibrated submanifolds and gauge theory: specifically, connections whose curvature satisfies 
  a natural constraint determined by the special holonomy group (so-called \emph{instantons}).  
For these reasons, calibrated submanifolds form a hot topic in current research, especially in the $\GG_2$ setting.

\begin{note}
These notes are primarily based on a lecture course the author gave at the LMS--CMI Research School ``An Invitation to Geometry and Topology via $\GG_2$'' at Imperial College London in July 2014.
\end{note}

\section{Minimal submanifolds}

We start by analysing the submanifolds which are critical points for the volume functional.  Let $N$ be a submanifold (without boundary) 
of a Riemannian manifold $(M,g)$ and let 
$F:N\times(-\epsilon,\epsilon)\rightarrow M$ be a variation of $N$ with compact support; i.e.~$F=\text{Id}$ outside a compact subset $\overline{S}$ of $N$ with $S$ open 
and $F(p,0)=p$ for all 
$p\in N$.  The vector field $X=\frac{\partial F}{\partial t}|_N$ is called the variation vector field (which will be zero outside of $\overline{S}$).  We then have the following definition.

\begin{dfn}
$N$ is \emph{minimal} if $\frac{\d}{\d t}\Vol(F(S,t))|_{t=0}=0$ for all variations $F$ with compact support $\overline{S}$ (depending on $F$).
\end{dfn}

\begin{remark}
Notice that we do not ask for $N$ to minimize volume: it is only stationary for the volume.  It could even be a maximum!
\end{remark}

\begin{ex}
A plane in $\R^n$ is minimal since any small variation  will have larger volume.
\end{ex}

\begin{ex}
Geodesics are locally length minimizing, so geodesics are minimal.  However, as an example,  the equator in $\mathcal{S}^2$ is minimal
but not length minimizing since we can deform it to a shorter line of latitude.  
\end{ex}

For simplicity let us suppose that $N$ is compact. We wish to calculate
$\frac{\d}{\d t}\Vol(F(N,t))|_{t=0}$.  Given local coordinates $x_i$ on $N$ we know that
$$\Vol(F(N,t))=\int_N\sqrt{\det\left(g\left(\frac{\partial F}{\partial x_i},\frac{\partial F}{\partial x_j}\right)\right)}\vol_N.$$ 
Let $p\in N$ and choose our coordinates $x_i$ to be normal coordinates at $p$: i.e.~so that $\frac{\partial F}{\partial x_i}(p,t)=e_i(t)$ satisfy $g(e_i(0),e_j(0))=\delta_{ij}$.  If $g_{ij}(t)=g(e_i(t),e_j(t))$ and $(g^{ij}(t))$ denotes the inverse of the matrix $(g_{ij}(t))$ then
we know that
$$\frac{\d}{\d t}\sqrt{\det(g_{ij}(t))}|_{t=0}=\frac{1}{2}\frac{\sum_{i,j}g^{ij}(t)g_{ij}'(t)}{\sqrt{\det(g_{ij}(t))}}|_{t=0}=\frac{1}{2}\sum_ig_{ii}'(0).$$

Now, if we let $\nabla$ denote the Levi-Civita connection of $g$, then
\begin{align*}
\frac{1}{2}\sum_ig_{ii}'(0)&=\frac{1}{2}\sum_i \frac{\d}{\d t}g\left(\frac{\partial F}{\partial x_i},\frac{\partial F}{\partial x_i}\right)|_{t=0}\\
&=\sum_ig(\nabla_Xe_i,e_i)\\
&=\sum_ig(\nabla_{e_i}X,e_i)=\Div_N(X)
\end{align*}
since $[X,e_i]=0$ (i.e.~the $t$ and $x_i$ derivatives commute).  Moreover, we see that
$$\Div_N(X)=\sum_ig(\nabla_{e_i}X,e_i)=\Div_N(X^{\rm T})-\sum_ig(X^{\perp},\nabla_{e_i}e_i)=\Div_N(X^{\rm T})-g(X,H)$$
(since $\nabla_{e_i}\big(g(X^{\perp},e_i)\big)=0$) where ${}^{\rm T}$ and ${}^\perp$ denote the tangential and normal parts and
$$H=\sum_{i}\nabla^{\perp}_{e_i}e_i$$
is the \emph{mean curvature vector}.  Overall we have the following.

\begin{thm} The first variation formula is
$$\frac{\d}{\d t}\Vol(F(N,t))|_{t=0}=\int_N \Div_N(X)\vol_N=-\int_N g(X,H)\vol_N.$$
\end{thm}

\begin{remark}
The $\Div_N(X^{\rm T})$ term does not appear in the first variation formula because its integral vanishes by the divergence theorem as $N$ is compact without boundary.  In general, it will still vanish since we assume for our variations that there exists a compact submanifold of $N$ with boundary which contains the support of $X^{\rm T}$ and so that $X^{\rm T}$ vanishes on the boundary.
\end{remark}

We deduce the following.

\begin{dfn}
$N$ is a \emph{minimal submanifold} if and only if $H=0$.
\end{dfn}

The equation $H=0$ is a \emph{second order nonlinear PDE}.  We can see this explicitly in the following simple case.  For a function $f:U\subseteq\R^{n-1}\rightarrow\R$ where $\overline{U}$ is compact, we see that if $N=\text{Graph}(f)\subseteq\R^n$ then the volume of $N$ is given by
$$\Vol(N)=\int_U\sqrt{1+|\nabla f|^2}\vol_U.$$
Any sufficiently small variation can be written $F(N,t)=\text{Graph}(f+th)$ for some $h:U\rightarrow\R$, so we can compute
\begin{align*}
\frac{\d}{\d t}\Vol(F(N,t))|_{t=0}&=\frac{\d}{\d t}|_{t=0}\int_U\sqrt{1+|\nabla f+t\nabla h|^2}\vol_U\\
&=\int_U\frac{\d}{\d t}|_{t=0}\sqrt{1+|\nabla f|^2+2t\langle\nabla f,\nabla h\rangle+t^2|\nabla h|^2}\vol_U\\
&=\int_U\frac{\langle \nabla f,\nabla h\rangle}{\sqrt{1+|\nabla f|^2}}\vol_U\\
&=-\int_U h\Div\left(\frac{\nabla f}{\sqrt{1+|\nabla f|^2}}\right)\vol_U.
\end{align*}
We therefore see that $N$ is minimal if and only if this vanishes for all $h$.  Hence, $\text{Graph}(f)$ is minimal in $\R^n$ if and only if 
$$\Div\left(\frac{\nabla f}{\sqrt{1+|\nabla f|^2}}\right)=0.$$
We see that we can write this equation as $\Delta f+Q(\nabla f,\nabla^2 f)=0$ where $Q$ consists of nonlinear terms (but linear in $\nabla^2f$).  Hence,
if we linearise this equation we just get $\Delta f=0$, so $f$ is harmonic.  In other words, the minimal 
 submanifold equation is a nonlinear equation whose linearisation is just Laplace's equation: this is an example of a nonlinear \emph{elliptic} PDE, which we 
 shall discuss further later.

\begin{ex}
A plane in $\R^n$ is trivially minimal because if $X,Y$ are any vector fields on the plane then $\nabla_X^{\perp}Y=0$ as the second fundamental form of a plane is zero.
\end{ex}

\begin{ex}
For curves $\gamma$, $H=0$ is equivalent to the geodesic equation $\nabla_{\dot{\gamma}}\dot{\gamma}=0$.
\end{ex}

The most studied minimal submanifolds (other than geodesics) are minimal surfaces in $\R^3$, since here the equation $H=0$ becomes a scalar equation on a surface, which is the simplest to analyse.  In general we would have a system of equations, which is more difficult to study.

\begin{ex}
The helicoid $M=\{(t\cos s,t\sin s,s)\in\R^3\,:\,s,t\in\R\}$ is a complete embedded minimal surface, discovered by Meusnier in 1776.
\end{ex}

\begin{ex}
The catenoid $M=\{(\cosh t\cos s,\cosh t\sin s ,t)\in\R^3\,:\,s,t,\in\R\}$ is a complete embedded minimal surface, discovered by Euler in 1744 and shown to be minimal 
by Meusnier in 1776.  The catenoid is another explicit example which is a critical point for volume but not minimizing.
\end{ex}

In fact the helicoid and the catenoid are locally isometric, and there is a 1-parameter family of locally isometric minimal surfaces deforming between the catenoid and helicoid: see, for example, \cite[Theorem 16.5]{Gray} for details.

It took about 70 years to find the next minimal surface, but now we know many examples of minimal surfaces in $\R^3$, as well as in other spaces by studying 
the nonlinear elliptic PDE given by the minimal surface equation.  The amount of 
literature in the area is vast, with key results including the proofs of the Lawson \cite{BLawson}, Willmore \cite{MNWillmore} and Yau \cite{IMNYau,MNYau,SongYau} Conjectures, and minimal surfaces have applications to major problems in geometry including the Positive Mass Theorem \cite{SYMass,SYMass2}, Penrose Inequality \cite{HIPenrose} and Poincar\'e Conjecture \cite{Perelman}.

\section{Introduction to calibrations}

As we have seen, minimal submanifolds are extremely important.  However there are two key issues.  
\begin{itemize}
\item Minimal submanifolds are defined by a second order nonlinear PDE system -- therefore they are hard to analyse.
\item Minimal submanifolds are only critical points for the volume functional, but we are often interested in minima for the volume functional -- we need 
a way to determine when this occurs.
\end{itemize}

We can help resolve these issues using the notion of calibration and calibrated submanifolds, introduced by Harvey--Lawson \cite{HL} in 1982.

\begin{dfn} \label{calibration.dfn}
A differential $k$-form $\eta$ on a Riemannian manifold $(M,g)$ is a \emph{calibration} if
\begin{itemize}
\item $\d\eta=0$ and
\item $\eta(e_1,\ldots,e_k)\leq 1$ for all unit tangent vectors $e_1,\ldots,e_k$ on $M$.
\end{itemize}
\end{dfn}

\begin{ex}
Any non-zero form with constant coefficients on $\R^n$ can be rescaled so that it is a calibration with at least one plane where equality holds.
\end{ex}

This example shows that there are many calibrations $\eta$, but the interesting question is: for which oriented planes $P=\Span\{e_1,\ldots,e_k\}$ does 
$\eta(e_1,\ldots,e_k)=1$?  
More importantly, can we find submanifolds $N$ so that this equality holds on each tangent space?  This motivates the next definition.

\begin{dfn}\label{calibsub.dfn} Let $\eta$ be a calibration $k$-form on $(M,g)$.
An oriented $k$-dimensional submanifold $N$ of $(M,g)$ is \emph{calibrated} by $\eta$ if $\eta|_N=\vol_N$, i.e.~if for all $p\in N$ we have 
$\eta(e_1,\ldots,e_k)=1$ for an oriented orthonormal basis $e_1,\ldots,e_k$ for $T_pN$.
\end{dfn}

\begin{ex}
Any oriented plane in $\R^n$ is calibrated.  If we change coordinates so that the plane $P$ is $\{x\in\R^n\,:\,x_{k+1}=\ldots=x_n=0\}$ (with the obvious orientation) then 
$\eta=\d x_1\wedge\ldots\wedge \d x_k$ is a calibration and $P$ is calibrated by $\eta$.
\end{ex}

Notice that the calibrated condition is now an algebraic condition on the tangent vectors to $N$, so being calibrated is a \emph{first order nonlinear PDE}.
We shall motivate these definitions further later, but for now we make the following observation.

\begin{thm}\label{calibmin.thm}
Let $N$ be a calibrated submanifold.  Then $N$ is minimal and, moreover, if 
$F$ is any variation with compact support $\overline{S}$ then $\Vol(F(S,t))\geq\Vol(S)$; i.e.~$N$ is volume-minimizing.  In particular, if $N$ is compact then $N$ is volume-minimizing in its homology class.
\end{thm}

\begin{proof}
Suppose that $N$ is calibrated by $\eta$ and suppose for simplicity that $N$ is compact.  We will show that $N$ is homologically volume-minimizing.

Suppose that $N'$ is homologous to $N$.   Then there exists a compact manifold $K$ with boundary $-N\cup N'$ and, since $\d\eta=0$, we have by Stokes' Theorem that
$$0=\int_K\d\eta=\int_{N'}\eta-\int_N\eta.$$ 
We deduce that
\begin{align*}
\Vol(N)&=\int_N\eta=\int_{N'}\eta\leq\Vol(N').
\end{align*}
We then have the result by the definition of minimal submanifold.
\end{proof}

We conclude this introduction with the following elementary result.

\begin{prop}
There are no compact calibrated submanifolds in $\R^n$.
\end{prop}  

\begin{proof}
Suppose that $\eta$ is a calibration and $N$ is compact and calibrated by $\eta$.  Then $\d\eta=0$ so by the Poincar\'e Lemma $\eta=\d\zeta$, and hence
$$\Vol(N)=\int_N\eta=\int_N\d\zeta=0$$
by Stokes' Theorem.
\end{proof}

Although there are many calibrations, having calibrated submanifolds greatly restricts the calibrations you want to consider.  
The calibrations which have calibrated submanifolds have special significance and there is a particular connection with special holonomy, due to the following 
observations.

Let $G$ be the holonomy group of a Riemannian metric $g$ on an $n$-manifold $M$.  
Then $G$ acts on the $k$-forms on $\R^n$, so suppose that $\eta_0$ is a $G$-invariant $k$-form.  We can 
always rescale $\eta_0$ so that $\eta_0|_P\leq \vol_P$ for all oriented $k$-planes $P$ and equality holds for at least one $P$.  Since $\eta_0$ is $G$-invariant, 
if $P$ is calibrated then so is $\gamma\cdot P$ for any $\gamma\in G$, which usually means we have quite a few calibrated planes.  
We know by the \emph{holonomy principle} (see, for example, \cite[Proposition 2.5.2]{Joycebook2})
that we then get a parallel $k$-form $\eta$ on $M$ which is identified with $\eta_0$ at every point.  Since $\nabla\eta=0$, we have $\d\eta=0$ 
and hence $\eta$ is a calibration.  Moreover, we have a lot of calibrated tangent planes on $M$, so we can hope to find calibrated submanifolds.

\section{Complex submanifolds}

We would now like to address the question: where does the calibration condition come from?  The answer is from \emph{complex geometry}.  On $\R^{2n}=\C^n$ 
with coordinates $z_j=x_j+iy_j$, we have the complex structure $J$ and the distinguished K\"ahler 2-form 
$$\omega=\sum_{j=1}^n\d x_j\wedge \d y_j=\frac{i}{2}\sum_{j=1}^n\d z_j\wedge\d \overline{z}_j.$$
More generally we can work with a \emph{K\"ahler manifold} $(M,J,\omega)$.  Our first key result is the following.

\begin{thm}\label{cxcalib.thm} On a K\"ahler manifold $(M,J,\omega)$, 
$\frac{\omega^k}{k!}$ is a calibration whose calibrated submanifolds are the complex $k$-dimensional submanifolds: i.e.~submanifolds $N$ 
such that $J(T_pN)=T_pN$ for all $p\in N$.  
\end{thm}

Since $\d\omega^k=k\d\omega\wedge\omega^{k-1}=0$, Theorem \ref{cxcalib.thm} follows immediately from the following result.

\begin{thm}[Wirtinger's inequality]\label{Wirtinger.thm}  For any unit vectors $e_1,\ldots,e_{2k}\in\C^n$, $$\frac{\omega^k}{k!}(e_1,\ldots,e_{2k})\leq 1$$  with equality if and only if 
$\Span\{e_1,\ldots,e_{2k}\}$ is a complex $k$-plane  in $\C^n$.
\end{thm}

Before proving this we make the following observation.

\begin{lem}\label{starcab.lem}
If $\eta$ is a calibration and $*\eta$ is closed then $*\eta$ is a calibration.  Moreover an oriented tangent plane $P$ is calibrated by $\eta$ if and only if there is an orientation on the orthogonal complement $P^{\perp}$ so that it is calibrated by $*\eta$.
\end{lem}

\begin{proof}
Suppose that $\eta$ is a calibration $k$-form on $(M,g)$ with $\d\!*\!\eta=0$.  Let $p\in M$.  Take any $n-k$ orthonormal tangent vectors $e_{k+1},\ldots,e_n$ at $p$.  Then there exist $e_1,\ldots,e_k\in T_pM$ so that $\{e_1,\ldots,e_n\}$ is an oriented orthonormal basis
for $T_pM$.  Since $\{e_1,\ldots,e_n\}$ is an oriented orthonormal basis, we can use the definition of the Hodge star to calculate 
$$*\eta(e_{k+1},\ldots,e_n)=\eta(e_1,\ldots,e_k)\leq 1.$$
Hence $*\eta$ is a calibration by Definition \ref{calibration.dfn}.  Moreover,  the oriented plane $P=\Span\{e_{k+1},\ldots,e_n\}$ is calibrated by $*\eta$ if and only if there is an orientation on $\Span\{e_1,\ldots,e_k\}=P^{\perp}$ so that it is calibrated 
by $\eta$, since $\eta(e_1,\ldots,e_k)=\pm *\eta(e_{k+1},\ldots,e_n)=\pm 1$.
\end{proof}

We can now prove Wirtinger's inequality.

\begin{proof}[Proof of Theorem \ref{Wirtinger.thm}.]
We see that $|\frac{\omega^k}{k!}|^2=\frac{n!}{k!(n-k)!}$ and $\vol_{\C^n}=\frac{\omega^n}{n!}$ so $*\frac{\omega^k}{k!}=\frac{\omega^{n-k}}{(n-k)!}$.  Hence, 
by Lemma \ref{starcab.lem}, it is enough to study the case where $k\leq \frac{n}{2}$.

Let $P$ be any $2k$-plane in $\C^n$ with $2k\leq n$.  We shall find a canonical form for $P$.  First consider $\langle Ju,v\rangle$ for orthonormal vectors $u,v\in P$.  This 
must have a maximum, so let $\cos\theta_1=\langle Ju,v\rangle$ be this maximum realised by some orthonormal vectors $u,v\in P$, where $0\leq\theta_1\leq\frac{\pi}{2}$.  

Suppose that $w\in P$ is a unit vector orthogonal to 
$\Span\{u,v\}$, where $\cos\theta_1=\langle Ju,v\rangle$.  The function
$$f_w(\theta)=\langle Ju,\cos\theta v+\sin\theta w\rangle$$
has a maximum at $\theta=0$ so $f_w'(0)=\langle Ju,w\rangle =0$.  Similarly we have that $\langle Jv,w\rangle =0$, and thus $w\in\Span\{u,v,Ju,Jv\}^{\perp}$.

We then have two cases.  If $\theta_1=0$ then $v=Ju$ so we can set $u=e_1,v=Je_1$ and see that $P=\Span\{e_1,Je_1\}\times Q$ where $Q$ is a $2(k-1)$-plane in
$\C^{n-1}=\Span\{e_1,Je_1\}^{\perp}$.  If $\theta_1\neq 0$ we have that $v=\cos\theta_1 Ju+\sin\theta_1 w$ where $w$ is a unit vector 
orthogonal to $u$ and $Ju$, so we can let $u=e_1$, $w=e_2$ and see that $P=\Span\{e_1,\cos\theta_1 Je_1+\sin\theta_1 e_2\}\times Q$ 
where $Q$ is a $2(k-1)$-plane in $\C^{n-2}=\Span\{e_1,Je_1,e_2,Je_2\}^{\perp}$.  

Proceeding by induction we see that we have an oriented basis $\{e_1,Je_1,\ldots,e_n,Je_n\}$ for $\C^n$ so that
$$P=\Span\{e_1,\cos\theta_1 Je_1+\sin\theta_1 e_2,\ldots,e_{2k-1},\cos\theta_k Je_{2k-1}+\sin\theta_k e_{2k}\},$$ where
$0\leq\theta_1\leq\ldots\leq\theta_{k-1}\leq\frac{\pi}{2}$ and $\theta_{k-1}\leq\theta_k\leq\pi-\theta_{k-1}$.

Since we can write $\omega=\sum_{j=1}^n e^j\wedge Je^j$ we see that $\frac{\omega^k}{k!}$ restricts to $P$ to give a product of $\cos\theta_j$ which is 
certainly less than or equal to $1$.  Moreover, equality holds if and only if all of the $\theta_j=0$ which means that $P$ is complex.
\end{proof}

Putting together Theorem \ref{cxcalib.thm} and Theorem \ref{calibmin.thm} yields the following.

\begin{cor}
$\!\!$Compact complex submanifolds of K\"ahler manifolds are homologically volume-minimizing.
\end{cor}

We know that complex submanifolds are defined by holomorphic functions; i.e.~solutions to the Cauchy--Riemann equations, which are a first-order PDE system, as one would expect for calibrated submanifolds.

\begin{ex}
$N=\{(z,\frac{1}{z})\in\C^2\,:\,z\in\C\setminus\{0\}\}$ is a complex curve in $\C^2$, and thus is calibrated.  %This example will appear again later.
\end{ex}

\begin{ex}
An important non-trivial example of a K\"ahler manifold is $\C\P^n$, where the zero set of a system of polynomial equations defines a (possibly singular) complex submanifold.
\end{ex}

\section{Special Lagrangians}

Complex submanifolds are very familiar, but can we find any other interesting classes of calibrated submanifolds?  The answer is that indeed we can, particularly
when the manifold has special holonomy.  We begin with the case of holonomy $\SU(n)$ -- so-called \emph{Calabi--Yau manifolds}.  The model example for Calabi--Yau manifolds is $\C^n$ with complex structure $J$, K\"ahler form $\omega$ and holomorphic volume form 
$$\Upsilon=\d z_1\wedge\ldots\wedge\d z_n,$$
if $z_1,\ldots,z_n$ are complex coordinates on $\C^n$.

\begin{thm}
Let $M$ be a Calabi--Yau manifold with holomorphic volume form $\Upsilon$.  Then $\Ree(e^{-i\theta}\Upsilon)$ is a calibration for any $\theta\in\R$.
\end{thm}

Since $\d\Upsilon=0$, the result follows immediately from the following.

\begin{thm}\label{Lagcalib.thm}
On $\C^n$, $|\Upsilon(e_1,\ldots,e_n)|\leq 1$ for all unit vectors $e_1,\ldots,e_n$ with equality if and only if $P=\Span\{e_1,\ldots,e_n\}$ is 
a Lagrangian plane, i.e.~$P$ is an $n$-plane such that $\omega|_P\equiv 0$.
\end{thm}

\begin{proof}
Let $e_1,\ldots,e_n$ be the standard basis for $\R^n$ and let $P$ be an $n$-plane in $\C^n$.  There exists $A\in\GL(n,\C)$ so that $f_1=Ae_1,\ldots,f_n=Ae_n$ is an orthonormal basis for $P$.  Then 
$\Upsilon(Ae_1,\ldots,Ae_n)=\det_{\C}(A)$ so 
$$|\Upsilon(f_1,\ldots,f_n)|^2=|\dett_{\C}(A)|^2=|\dett_{\R}(A)|=|f_1\wedge Jf_1\wedge\ldots\wedge f_n\wedge J f_n|\leq |f_1||Jf_1|\ldots|f_n||Jf_n|=1$$
with equality if and only if $f_1,Jf_1,\ldots,f_n, Jf_n$ are orthonormal.  However, this is exactly equivalent to the Lagrangian condition, since
$\omega(u,v)=g(Ju,v)$ so $\omega|_P\equiv 0$ if and only if $JP=P^{\perp}$.    
\end{proof}

\begin{dfn}
A submanifold $N$ of $M$ calibrated by $\Ree(e^{-i\theta}\Upsilon)$ is called \emph{special Lagrangian} with phase $e^{i\theta}$.  If $\theta=0$ we say 
that $N$ is simply special Lagrangian.  By Theorem \ref{Lagcalib.thm}, we see that $N$ is special Lagrangian if and only if $\omega|_N\equiv 0$ (i.e.~$N$ is Lagrangian) 
and $\Imm\Upsilon|_N\equiv 0$ (up to a choice of orientation so that $\Ree\Upsilon|_N>0$).
\end{dfn}

\begin{ex}
Consider $\C=\R^2$ with coordinates $z=x+iy$, complex structure $J$ given by $Jw=iw$, K\"ahler form $\omega=\d x\wedge\d y=\frac{i}{2}\d z\wedge\d\overline{z}$ and holomorphic volume form $\Upsilon=\d z=\d x+i\d y$.  We want to consider the special Lagrangians in $\C$, which are 1-dimensional submanifolds or curves $N$ in 
$\C=\R^2$.

Since $\omega$ is a 2-form, it vanishes on any curve in $\C$.  Hence every curve in $\C$ is Lagrangian.  For $N$ to be special Lagrangian with phase $e^{i\theta}$ 
we need that 
$$\Ree(e^{-i\theta}\Upsilon)=\cos\theta\d x+\sin\theta\d y$$
is  the volume form on $N$, or equivalently that 
$$\Imm(e^{-i\theta}\Upsilon)=\cos\theta\d y-\sin\theta\d x$$
vanishes on $N$.  This means that $\cos\theta\partial_x+\sin\theta\partial_y$ is everywhere a unit tangent vector to $N$, so $N$ is a straight line 
given by $N=\{(t\cos\theta ,t\sin\theta )\in\R^2\,:\,t\in\R\}$ (up to translation), so it makes an angle $\theta$ with the $x$-axis, hence motivating the term 
``phase $e^{i\theta}$''.  

Notice that this result is compatible with the fact that special Lagrangians are minimal, and hence must be geodesics in $\R^2$; i.e.~straight lines.
\end{ex}

\begin{ex}
Consider $\C^2=\R^4$.  We know that $\omega=\d x_1\wedge\d y_1+\d x_2\wedge\d y_2$.  Since $\Upsilon=\d z_1\wedge \d z_2=(\d x_1+i\d y_1)\wedge (\d x_2+i\d y_2)$, we also know that $\Ree\Upsilon=\d x_1\wedge\d x_2+\d y_2\wedge\d y_1$, which 
looks somewhat similar.  In fact, if we let $J'$ denote the complex structure given by $J'(\partial_{x_1})=\partial_{x_2}$ and $J'(\partial_{y_2})=\partial_{y_1}$, 
then $\Ree\Upsilon=\omega'$, the K\"ahler form corresponding the complex structure $J'$.  Hence special Lagrangians in $\C^2$ are complex curves for a different 
complex structure.  

In fact, we have a hyperk\"ahler triple of complex structures $J_1,J_2,J_3$, where $J_1=J$ is the standard one and $J_3=J_1J_2=-J_2J_1$ so that $J_1=J_2J_3=-J_3J_2$ and $J_2=J_3J_1=-J_1J_3$, and the corresponding K\"ahler forms are 
$\omega=\omega_1$, $\omega_2$, $\omega_3$ which are orthogonal and the same length with $\Upsilon=\omega_2+i\omega_3$. 
\end{ex}

This shows we should only consider complex dimension 3 and higher to find new calibrated submanifolds.

\begin{ex} Let $f:\R^n\rightarrow \R^n$ be a smooth function and let $N=\text{Graph}(f)\subseteq\R^{2n}=\C^n$. We want to see when $N$ is special Lagrangian.  
 We see that tangent vectors to $N$ are given by 
$$e_1+i\nabla_{e_1}f,\ldots,e_n+i\nabla_{e_n}f.$$
Hence $N$ is Lagrangian if and only if
$$\omega(e_j+i\nabla_{e_j}f,e_k+i\nabla_{e_k}f)=\nabla_{e_k}f_j-\nabla_{e_j}f_k=0$$
for all $j,k$.  Since $\R^n$ is simply connected, this occurs if and only if there exists $F$ such that $f_j=\nabla_{e_j}F$; i.e.~$f=\nabla F$.

Recall that $\Upsilon=\d z_1\wedge\ldots\wedge \d z_n$.
  We know that $N$ is special Lagrangian if and only if $N$ is Lagrangian and $\Imm\Upsilon$ vanishes on $N$.  Now
$$\Upsilon(a_1+ib_1,\ldots,a_n+ib_n)=\dett_{\C}(A+iB)$$
where $A,B$ are the matrices with columns $a_i,b_j$ respectively.  Hence 
$$\Upsilon(e_1+i\nabla_{e_1}\nabla F,\ldots,e_n+i\nabla_{e_n}\nabla F)=\dett_{\C}(I+i\text{Hess}\,F),$$
 where $\Hess F=(\frac{\partial^2 F}{\partial x_i\partial x_j})$.  

Therefore $N=\text{Graph}(f)$ is special Lagrangian (up to a choice of orientation) if and only if $f=\nabla F$ and
$$\Imm\dett_{\C}(I+i\Hess F)=0.$$
 
If $n=2$, 
$$I+i\text{Hess}\,F=\left(\begin{array}{cc} 1+iF_{xx} & iF_{xy} \\ iF_{yx} & 1+iF_{yy}\end{array}\right).$$
Therefore, the determinant gives
$$1-F_{xx}F_{yy}+F_{xy}^2+i(F_{xx}+F_{yy}),$$
then the imaginary part is $F_{xx}+F_{yy}$.  Therefore, $N$ is special Lagrangian if and only if $\Delta F=0$.  

As we know, a graph in $\C^2$ of $f=u+iv:\C\to\C$ is a complex surface if and only if $u+iv$ is holomorphic, which implies that $u,v$ are harmonic.  We know that 
special Lagrangians in $\C^2$ are complex surfaces for a different complex structure, so this is expected.

If $n=3$,
$$I+i\text{Hess}\, F=\left(\begin{array}{ccc} 1+iF_{xx} & i F_{xy} & iF_{xz} \\ iF_{yx} & 1+iF_{yy} & iF_{yz} \\ iF_{zx} & iF_{zy} & 1+iF_{zz} \end{array}\right).$$
Hence, 
\begin{align*}
\Imm\dett_{\C}(I+i\Hess F)=&\;F_{xx}+F_{yy}+F_{zz}\\
&-F_{xx}(F_{yy}F_{zz}-F_{yz}^2)-F_{xy}(F_{yz}F_{zx}-F_{xy}F_{zz})-F_{zx}(F_{xy}F_{yz}-F_{yy}F_{zx}).
\end{align*}
Therefore, $N$ is special Lagrangian if and only if
\begin{align*}
-\Delta F&=F_{xx}+F_{yy}+F_{zz}\\
&=F_{xx}(F_{yy}F_{zz}-F_{yz}^2)-F_{xy}(F_{xy}F_{zz}-F_{yz}F_{zx})+F_{zx}(F_{xy}F_{yz}-F_{yy}F_{zx})\\
&=\dett\Hess F.
\end{align*} 
\end{ex}

We now wish to describe some very important examples of special Lagrangians, which are asymptotic to pairs of planes. 

\begin{ex}
$\SU(n)$ acts transitively on the space of special Lagrangian planes with isotropy $\SO(n)$.  So any special Lagrangian plane is given by $A\cdot\R^n$ for $A\in\SU(n)$ where $\R^n$ is the standard real $\R^n$ in $\C^n$.

Given $\theta=(\theta_1,\ldots,\theta_n)$ we can define a plane $P(\theta)=\{(e^{i\theta_1}x_1,\ldots,e^{i\theta_n}x_n)\in\C^n\,:\,(x_1,\ldots,x_n)\in\R^n\}$
 (where we can swap orientation).  We see that $P(\theta)$  
is special Lagrangian if and only if $\Ree\Upsilon|_P=\pm\cos(\theta_1+\ldots+\theta_n)=1$ so that $\theta_1+\ldots+\theta_n\in\pi\Z$.  Given any 
$\theta_1,\ldots,\theta_n\in(0,\pi)$ with $\theta_1+\ldots+\theta_n=\pi$, there exists a special Lagrangian $N$ (called a \emph{Lawlor neck}) 
asymptotic to $P(0)\cup P(\theta)$: see, for example, \cite[Example 8.3.15]{Joycebook2} or $\S$\ref{s:angle} for details.  It is diffeomorphic to $\mathcal{S}^{n-1}\times\R$.  By rotating coordinates we have a special Lagrangian with phase $i$ 
asymptotic to $P(-\frac{\theta}{2})\cup P(\frac{\theta}{2})$.

The simplest case is when $\theta_1=\ldots=\theta_n=\frac{\pi}{n}$: here $N$ is called the \emph{Lagrangian catenoid}.  When $n=2$, under a coordinate change
the Lagrangian catenoid becomes the complex curve $\{(z,\frac{1}{z})\in\C^2\,:\,z\in\C\setminus\{0\}\}$ that we saw before.  When $n=3$, the only possibilities
 for the angles are $\sum_i\theta_i=\pi,2\pi$, but if $\sum_i\theta_i=2\pi$ we can rotate coordinates and change the order of the planes so that $P(0)\cup P(\theta)$ becomes $P(0)\cup P(\theta')$ where $\sum_i\theta_i'=\pi$.  Hence, given any pair of transverse special Lagrangian planes in $\C^3$, there exists a Lawlor neck 
 asymptotic to their union.
\end{ex}

\begin{remark}  Using complex geometry it is easy to classify all of the smooth special Lagrangians in $\C^2$ asymptotic to a pair of transverse planes, and one sees that the Lawlor necks in $\C^2$ are the unique exact special Lagrangians with this property. 
It is now known that the Lawlor necks are the unique smooth exact special Lagrangian asymptotic to a pair of planes in all dimensions \cite{IJO}.
\end{remark}

We can find special Lagrangians in Calabi--Yau manifolds using the following easy result.

\begin{prop}\label{SLfixedpt.prop}
Let $(M,\omega,\Upsilon)$ be a Calabi--Yau manifold  and let $\sigma:M\rightarrow M$ be such that $\sigma^2=\text{\emph{Id}}$, $\sigma^*(\omega)=-\omega$, $\sigma^*(\Upsilon)=\overline{\Upsilon}$.  Then $\text{Fix}(\sigma)$ is special Lagrangian, if it is non-empty.
\end{prop}

%\begin{proof}
%If $p\in N=\text{Fix}(\sigma)$ then $\omega|_{T_pN}=0$ since $\sigma^*(\omega)=-\omega$ and $\Imm\Omega|_{T_pN}=0$ since $\sigma^*\Omega=\overline{\Omega}$.   
%\end{proof}

\begin{ex}
Let $X=\{[z_0,\ldots,z_4]\in\C\P^4\,:\,z_0^5+\ldots+z_4^5=0\}$ (the \emph{Fermat quintic}) with its Calabi--Yau structure (which exists  
by Yau's solution of the Calabi conjecture since the first Chern class of $X$ vanishes).  Let $\sigma$ be the restriction of complex conjugation on $\C\P^4$ to $X$.  
Then the fixed point set of $\sigma$, which is the real locus in $X$, is a special Lagrangian 3-fold (if it is non-empty).  (There is a subtlety here: $\sigma$ is certainly an anti-holomorphic isometric involution for the induced metric on $X$, but this is \emph{not} the same as the Calabi--Yau metric on 
$X$.  Nevertheless, it is the case that $\sigma$ satisfies the conditions of Proposition \ref{SLfixedpt.prop}.)
\end{ex}

\begin{ex}
There exists a Calabi--Yau metric on $T^*\mathcal{S}^n$ (the Stenzel metric \cite{Stenzel}) so that the base $\mathcal{S}^n$ is special Lagrangian:   When $n=2$ this is a hyperk\"ahler 
metric called the Eguchi--Hanson metric \cite{EguchiHanson}.
\end{ex}

%\begin{center}
%{\large Lecture 3: Constructing calibrated submanifolds, deformations of special Lagrangians}
%\end{center}

\section{Constructing calibrated submanifolds}

It is easy to construct complex submanifolds in K\"ahler manifolds algebraically.  Constructing other calibrated submanifolds is much more 
challenging because one needs to solve a nonlinear PDE, even in Euclidean space.  
There are  approaches in Euclidean space and other simple spaces which have involved reducing the problem to ODEs or other problems which do not require PDE (for example, algebraic methods).  For example, 
we have the following methods, which you can find out more about in \cite{Joycebook2} or the references provided.
\begin{itemize}
\item  Symmetries/evolution equations \cite{Goldstein, HL, Haskins.SLcones1, Ionel.Minoo, Joyce.evol, Joyce.quadrics, Joyce.sym, Joyce.U1fib, Joyce.U1a, Joyce.U1b, Joyce.U1c, Lotayevol, Lotaysym}.
%Search for examples $N$ with large symmetry groups $G$, so that $N$ is a 1-parameter family of $G$-orbits, which becomes a system of ODEs.  More generally, one can start from a submanifold of 1 dimension lower (possibly satisfying a constraint, e.g.~$\omega$ vanishes on it in the special Lagrangian case) and evolve it into a 1-parameter family whose total space is calibrated.  Joyce  has a separate analytic construction of $\UU(1)$-invariant special Lagrangian 3-folds.
\item Use of integrable systems to study calibrated cones \cite{Carberry, Carberry.McIntosh, Haskins.SLcones2, Joyce.int, McIntosh}. 
\item Calibrated cones and ruled smoothings of these cones \cite{BryantOct, Bryant.SL, Fox.coass, Fox.Cayley, Joyce.ruled, Lotayevol,  Lotay2R, LotayLag}. 
%by calibrated submanifolds fibered by lines/2-planes, using exterior differential systems and reduction to linear PDE (in fact, the Cauchy--Riemann equations).
\item  Vector sub-bundle constructions \cite{Ionel.Kar.Minoo, Kar.Leung, Kar.Minoo}.
\item Classification of calibrated submanifolds satisfying pointwise constraints on their second fundamental form \cite{Bryant.SL, Fox.thesis, Ionel.SL, LotayLag, Lotayassoc}.
\end{itemize}

However, an important direction which has borne fruit in calibrated geometry and special holonomy recently has been to study the nonlinear PDE head on, especially 
by perturbative and gluing methods.   %To discuss this, we need an aside on the types of nonlinear PDE we want to solve.  
%We start with some basics on elliptic PDE: Laplace's equation is a prototypical example.

%\begin{dfn}
%If $f$ is a function on a manifold $M$ then a differential operator $P$ of order $m$ on $f$ is given by $$P(f)(x)=P(x,\nabla f(x),\ldots,\nabla^mf(x))$$ for some function $P$.  If $P$ is a linear differential operator (i.e.~the function $P$ is linear in $f$), then we can write it in local coordinates as $$P(f)=\sum_{k=0}^m a_{i_1\ldots i_k}\nabla_{i_1}\ldots \nabla_{i_k} f$$ for functions $a_{i_1\ldots i_k}$.  For example, a linear second order differential operator on $\R^n$ is given by $$P(f)=\sum a_{ij}\frac{\partial^2 f}{\partial x_i\partial x_j}+b_i\frac{\partial f}{\partial x_i}+cf$$ for functions $a_{ij},b_i,c$.
%
%Given vector bundles $V,W$ over $M$ and a section $v$ of $V$, a differential operator $P$ from sections of $V$ to sections of $W$ takes $v$ is given by $$P(v)(x)=P(x,\nabla v(x),\ldots,\nabla^mv(x))$$ and linear differential operators are given locally by $$P(v)(x)=\sum_{k=0}^m a_{i_1\ldots i_k}(x)\nabla_{i_1\ldots i_k}v(x)$$ where now $a_{i_1\ldots i_k}$ takes values in $V^*\otimes W$ (a linear map from $V$ to $W$) so that it maps $\nabla_{i_1\ldots i_k}v(x)\in V_x$ to $P(v)(x)\in W_x$. 
%
%Given any differential operator we can always linearise it: the linearisation of $P$ at $u$ is given by $$L_uP(v)=\lim_{t\rightarrow 0}\frac{P(u+tv)-P(u)}{t},$$  which is a linear differential operator.
%\end{dfn}

We want to solve nonlinear PDE, so how do we tackle this?  The idea is to use the linear case to help.  Suppose we are on a compact manifold $N$ and recall the 
theory of linear \emph{elliptic} operators $L$ of order $l$ on $N$, including: 
\begin{itemize}
\item the definition of ellipticity of $L$ via the \emph{principal symbol} $\sigma_L$ (which encodes the highest order  derivatives in the operator) being an isomorphism;
\item the use of \emph{H\"older spaces} $C^{k,a}$ to give elliptic regularity theory (so-called \emph{Schauder theory}), namely that if $w\in C^{k,a}$ and 
$Lv=w$ then $v\in C^{k+l,a}$ and there is a universal constant $C$ so that
$$\|v\|_{C^{k+l,a}}\leq C(\|Lv\|_{C^{k,a}}+\|v\|_{C^0})$$
(and we can drop the $\|v\|_{C^0}$ term if $v$ is orthogonal to $\Ker L$);
\item the adjoint operator $L^*$ and that $\sigma_{L^*}=(-1)^l\sigma_L^*$ so that $L^*$ is elliptic if and only if $L$ is elliptic; and
\item the Fredholm theory of $L$, namely that $\Ker L$ (and hence $\Ker L^*$)  is finite-dimensional, and we can solve $Lv=w$ if and only if $w\in (\Ker L^*)^{\perp}$.
\end{itemize}

We shall discuss this in a model example which we shall use throughout this section.

\begin{ex}
The Laplacian on functions is given by $\Delta f=\d^*\d f$ which in normal coordinates at a point is given by $f\mapsto -\sum_i \frac{\partial^2f}{\partial x_i^2}$, so it is a linear second order differential operator.  
We see that its principal symbol is $\sigma_{\Delta}(x,\xi)f= -|\xi|^2f$ which is an isomorphism for $\xi\in T^*_xN\setminus\{0\}$, so $\Delta$ is elliptic.  We therefore have 
that if $h\in C^{k,a}(N)$ and $\Delta f=h$ then $f\in C^{k+2,a}(N)$, and we have an estimate
  $$\|f\|_{C^{k+2,a}}\leq C(\|\Delta f\|_{C^{k,a}}+\|f\|_{C^0}).$$
  We also know that $\Delta^*=\Delta$ and $\Ker\Delta$ is given by the constant functions (since if $f\in\Ker\Delta$ then 
$$0=\langle f,\Delta f\rangle_{L^2}=\langle f,\d^*\d f\rangle_{L^2}=\|\d f\|_{L^2}^2$$
so $\d f=0$).  Hence, we can solve $\Delta f=h$ if and only if $h$ is orthogonal to the constants, i.e.~$\int_N h\vol_N=0$.

  The operator defining the minimal graph equation for a hypersurface is 
$$P(f)=-\Div\left(\frac{\nabla f}{\sqrt{1+|\nabla f|^2}}\right),$$ which is a nonlinear second order operator whose linearisation $L_0P$ at $0$ is $\Delta$.  
Thus $P$ is a nonlinear elliptic operator at $0$.  If we linearise $P$ at $f_0$ we find a more complicated expression depending on $f_0$, but it is 
still a perturbation of the Laplacian.
\end{ex}

Suppose we are on a compact manifold $N$ and we want to solve $P(f)=0$ where $P$ is the minimal graph operator on functions $f$.  Let us consider 
regularity for $f$.
%$P(f)=\Delta f+Q(\nabla f,\nabla^2f)=0$ (so $L_0P=\Delta$) where $f$ is a function and $Q=Q(f)$ is linear in $\nabla^2f$ (i.e.~like the minimal submanifold equation).  
We can re-arrange $P(f)=0$ by taking all of the second derivatives to one side as:
$$R(x,\nabla f(x))\nabla^2 f(x)=E(x,\nabla f(x))$$
where $x\in N$. %and $R,E:C^{k+1,a}\rightarrow C^{k,a}$.  
  Since $L_0P=\Delta$ is elliptic and ellipticity is an open condition we know that the operator $L_f$ (depending on $f$) given by
$$L_f(h)(x)=R(x,\nabla f(x))\nabla^2h(x)$$ is a \emph{linear} elliptic operator
whenever $\|\nabla f\|_{C^0}$ is small, in particular if $\|f\|_{C^{1,a}}$ is sufficiently small. The operator $L_f$ does not have smooth coefficients, but 
if $f\in C^{k,a}$ then the coefficients $R\in C^{k-1,a}$.

Suppose that $f\in C^{1,a}$ and $\|f\|_{C^{1,a}}$ is small with $P(f)=0$. Then $L_f(f)=E(f)$ and $L_f$ is a linear \emph{second order} elliptic operator with coefficients in $C^{0,a}$ and 
 $E(f)$ is in $C^{0,a}$.  
 So by  elliptic regularity we can deduce that $f\in C^{2,a}$.  We have gained one degree of regularity, so we can ``bootstrap'', 
i.e.~proceed by induction and deduce that any $C^{1,a}$ solution to $P(f)=0$ is smooth.  

\begin{ex}
$C^{1,a}$-minimal submanifolds (and thus calibrated submanifolds) are \emph{smooth}.  
\end{ex}

\begin{remark}
More sophisticated techniques can be used to deduce that $C^1$-minimal submanifolds are real analytic \cite{Morrey}.  Notice that elliptic regularity results are \emph{not} valid for $C^k$ spaces, so this result is not obvious.
\end{remark}

We can also arrange our simple equation $P(f)=0$ as $\Delta f+Q(\nabla f,\nabla^2f)=0$, where $Q$ is nonlinear but linear in $\nabla^2f$.  
If we know that $\int_N P(f)\vol_N=0$, i.e.~that $P(f)$ is 
orthogonal to the constants, then
%Now $\langle f,\Delta f\rangle=\|\nabla f\|^2$ so $\Delta f=0$ implies that $f$ is constant, so the kernel of $\Delta$ is $1$-dimensional and given by scalars.  
%Since $\Delta^*=\Delta$ and $\Delta$ is elliptic, we see that we can 
%solve $\Delta f=g$ if and only if $g$ is orthogonal to $\Ker\Delta^*$, which is the constant functions, i.e.~$\int_N g=0$, and the solution $f$ is unique up to 
%constants. %(Notice that $\int_N\Delta f\vol_N=\int_N\d*\d f=0$, so we could only solve $\Delta f=g$ when $\int_N g=0$.)  
%Therefore if we know that $\int_N P(f)=0$ for any function $f$,
 we can always solve $\Delta f_0=-Q(\nabla f,\nabla^2f)$.  We do know that $\int_NP(f)\vol_N=0$ since $P$ has a divergence form.  This means we are in the setting 
for implementing the Implicit Function Theorem for Banach spaces to conclude that we can always solve $P(f)=0$ for some $f$ near $0$, and $f$ will be smooth by our regularity argument above. In general, we will use the following.

\begin{thm}[Implicit Function Theorem] Let $X,Y$ be Banach spaces, let $U\ni 0$ be open in $X$, let $P:U\rightarrow Y$ with $P(0)=0$ and $L_0P:X\rightarrow Y$  surjective with finite-dimensional  kernel $K$.  

Then for some $U$, 
$P^{-1}(0)=\{u\in U\,:\, P(u)=0\}$ is a manifold of dimension $\dim K$.
Moreover, if we write $X=K\oplus Z$, $P^{-1}(0)=\text{\emph{Graph}}\, G$ for some map $G$ from an open set in $K$ to $Z$ with $G(0)=0$.
%If $L_0P$ is elliptic, $P^{-1}(0)$ consists of smooth elements.
\end{thm}

%\begin{remark}
%If we write $X=K\oplus Z$, $P^{-1}(0)=\text{Graph} G$ for some map $G$ from (an open set in) $K$ to $Z$.
%\end{remark}

%In our simple situation, suppose we can choose $X=C^{k+2,a}(M)$, $Y=\{g\in C^{k,a}(M)\,:\,\int_N g=0\}$, then the theorem above allows us to conclude we can solve 
%$P(f)=0$ for some $f$ near $0$, which is smooth by above and unique up to a constant.  
This gives us a way to describe all perturbations of a given 
calibrated submanifold, as we now see in the special Lagrangian case, due to McLean \cite{McLean}.

\begin{thm}\label{SLdef.thm}
Let $N$ be a compact special Lagrangian in a Calabi--Yau manifold $M$.  Then the moduli space of deformations of $N$ is a smooth manifold of dimension 
$b^1(N)$.
\end{thm}

\begin{remark}
One should compare this result to the deformation theory for complex submanifolds in K\"ahler manifolds.  There, one does not get that the moduli space is 
a smooth manifold: in fact, it can be singular, and one has \emph{obstructions} to deformations.  It is somewhat remarkable that special Lagrangian calibrated 
geometry enjoys a much better deformation theory than this classical calibrated geometry.  The deformation theory of embedded compact complex submanifolds in Calabi--Yau manifolds has recently been revisited using analytic techniques \cite{Moore.cpt}.  
%Embeddedness is important, since the result is false in the immersed setting.
\end{remark}

\begin{proof}
The tubular neighbourhood theorem gives us a diffeomorphism $\exp:S\subseteq\nu(N)\rightarrow T\subseteq M$ which maps the zero section to $N$; in other words, we can write any nearby submanifold to $N$ as the graph of a normal vector field on $N$.  We know that  $N$ is Lagrangian, so the complex structure $J$ gives an isomorphism 
between $\nu(N)$ and $TN$ and the metric gives an isomorphism between $TN$ and $T^*N$: $v\mapsto g(Jv,.)=\omega(v,.)=\alpha_v$. Therefore any deformation of $N$ in $T$ is given as the graph of a 1-form. 
%$N_v=\text{Graph}(v)$ for $v\in U$, so $\exp_v:N\rightarrow N_v$ is a diffeomorphism.  
%
%Now we want to characterise the condition that $N_v$ is special Lagrangian: this is that $\exp_v^*(\omega)=\exp_v^*(\Imm\Omega)=0$.  We want to linearise these 
%equations:
%$$\frac{\d}{\d t}\exp_{tv}^*(\omega)|_{t=0}=\mathcal{L}_v\omega=\d(v\lrcorner\omega)=\d\alpha_v$$
%and, as we shall see in the problem sheet,
%$$\frac{\d}{\d t}\exp_{tv}^*(\Imm\Omega)|_{t=0}=\mathcal{L}_v\Imm\Omega=\d(v\lrcorner\Imm\Omega)=d(*\alpha_v).$$
In fact, using the Lagrangian neighbourhood theorem, we can arrange that any $N'\in T$ is the graph of a 1-form $\alpha$, so that if $f_\alpha:N\rightarrow N_{\alpha}$ is the natural diffeomorphism then
$$f_{\alpha}^*(\omega)=\d\alpha\quad\text{and}\quad -*f_{\alpha}^*(\Imm\Upsilon)=F(\alpha,\nabla\alpha)=\d^*\alpha+Q(\alpha,\nabla\alpha),$$
where the second formula follows from a calculation using the special Lagrangian condition on $N$ and the fact that the ambient structure is Calabi--Yau.
Hence, $N_{\alpha}$ is special Lagrangian if and only if $P(\alpha)=(F(\alpha,\nabla\alpha),\d\alpha)=0$.  This means that infinitesimal special Lagrangian
 deformations are given by closed and coclosed 1-forms, which is the kernel of $L_0P$.

%If we let $P(\alpha)=(F(\alpha,\nabla\alpha),\d\alpha)$ then $P(\alpha)=0$ if and only if $N_{\alpha}$ is special Lagrangian. 
Since $\Imm\Upsilon=0$ on $N$ we have that $[\Imm\Upsilon]=0$ on $N_\alpha$, which means that $f_{\alpha}^*(\Imm\Upsilon)$ is exact. Thus $F(\alpha,\nabla\alpha)=-*f_{\alpha}^*(\Imm\Upsilon)$ is coexact and so
$$P:C^{\infty}(S)\rightarrow \d^*(C^{\infty}(T^*N))\oplus\d(C^{\infty}(T^*N))\subseteq C^{\infty}(\Lambda^0T^*N\oplus\Lambda^2T^*N).$$
If we let $X=C^{1,a}(T^*N)$, $Y=\d^*(C^{1,a}(T^*N))\oplus\d(C^{1,a}(T^*N))$ and $U=C^{1,a}(S)$ we can apply the Implicit Function Theorem if we know that 
$$L_0P:\alpha\in X\mapsto (\d^*\alpha,\d\alpha)\in Y$$
is surjective, i.e.~given $\d\beta+\d^*\gamma\in Y$ does there exist $\alpha$ such that $\d\alpha=\d\beta$ and $\d^*\alpha=\d^*\gamma$?  If we let 
$\alpha=\beta+\d f$ then we need $\Delta f=\d^*\d f=\d^*(\gamma-\beta)$.  Since $$\int_N\d^*(\gamma-\beta)\vol_N=\pm\int_N\d*(\gamma-\beta)=0$$ we can solve the equation for $f$, 
and hence $L_0P$ is 
surjective.

Therefore $P^{-1}(0)$ is a manifold of dimension $\dim\Ker L_0P=b^1(N)$ by Hodge theory.  Moreover, if $P(\alpha)=0$ then $N_{\alpha}$ is special Lagrangian, 
hence minimal and since $\alpha\in C^{1,a}$ we deduce that $\alpha$ is in fact smooth.
%$\d\alpha=0$ so we can write $\alpha=\beta+\d f$ where $\beta$ is harmonic (and hence smooth) and
% $\Delta f+Q(\d f,\nabla (\d f))=0$, so we can apply our bootstrapping argument to deduce that $f$ is smooth if $f\in C^{1,a}$ and thus $\alpha$ is smooth.  
%We can see this directly as $L_0P$ is overdetermined elliptic. 
\end{proof}

\begin{ex}
The special Lagrangian $\mathcal{S}^n$ in $T^*\mathcal{S}^n$ has $b^1=0$ and so is rigid.
\end{ex}

Observe that if we have a special Lagrangian $T^n$ in $M$ then $b^1(T^n)=n$ and, if the torus is close to flat then its deformations locally foliate
 $M$ (as there will be $n$ nowhere vanishing harmonic 1-forms), so we can hope to find special Lagrangian torus fibrations.  This cannot happen in compact manifolds without singular fibres,
but still motivates the SYZ conjecture in Mirror Symmetry. The deformation result also  motivates the following theorem \cite{Bryant.embed}.

\begin{thm}
Every compact oriented real analytic Riemannian 3-manifold can be isometrically embedded in a Calabi--Yau 3-fold as the fixed point set of an involution.
\end{thm}

\begin{remark} Theorem \ref{SLdef.thm} has also been extended to certain non-compact, singular and boundary settings, for example in \cite{Butscherdef, JoyceCS2, Pacini1}.
\end{remark}

%\begin{center}
%{\large Lecture 4: Gluing methods, associatives and coassociatives}
%\end{center}

Another well-known way to get a solution of a linear PDE from two solutions is simply to add them.  However, for  a nonlinear PDE $P(v)=0$ this will not 
work.  Intuitively, we can try to add two solutions to give us a solution $v_0$ for which $P(v_0)$ is small.  
Then we may try to perturb $v_0$ by $v$ to solve $P(v+v_0)=0$.  

 Geometrically, this occurs when we have two calibrated submanifolds $N_1,N_2$ and then glue them 
together to give a submanifold $N$ which is ``almost'' calibrated, then we deform $N$ to become calibrated.  If the two submanifolds $N_1,N_2$ are glued using a very long neck then one can imagine that $N$ is 
almost the disjoint union of $N_1,N_2$ and so close to being calibrated.  If instead one scales $N_2$ by a factor $t$ and then glues it into a singular point of 
$N_1$, we can again imagine that as $t$ becomes very small $N$ resembles $N_1$ and so again is close to being calibrated.  These two examples are in fact related, 
because if we rescale the shrinking $N_2$ to fixed size, then we get a long neck between $N_1$ and $N_2$ of length of order $-\log t$.  However, although these 
pictures are appealing, they also reveal the difficulty in this approach: as $t$ becomes small, $N$ 
becomes more ``degenerate'', giving rise to analytic difficulties which are encoded in the geometry of $N_1,N_2$ and $N$.

These ideas are used extensively in geometry, and particularly successfully in calibrated geometry e.g.~\cite{Butscher, HaskinsKapouleas, JoyceCS5, JoyceCS3, JoyceCS4, YILee, Lotaydesing, Lotaydesing2, Pacini2}.  A particular simple case is the following, which we will describe to show the basic idea of the gluing method.

\begin{thm}
Let $N$ be a compact connected 3-manifold and let $i:N\rightarrow M$ be a special Lagrangian immersion with tranverse self-intersection points in a Calabi--Yau
manifold $M$.  Then there exist embedded special Lagrangians $N_t$ such that $N_t\rightarrow N$ as $t\rightarrow 0$.
\end{thm}

\begin{remark}
One might ask about the sense of convergence here: for definiteness, we can say that $N_t$ converges to $N$ in the sense of currents; that is, if we have any 
compactly supported 3-form $\chi$ on $M$ then $\int_{N_t}\chi\rightarrow\int_N\chi$ as $t\rightarrow 0$.  However, all sensible notions of convergence of 
submanifolds will be true in this setting.
\end{remark}

\begin{proof}  Here we only provide a sketch of the proof: see, for example, \cite[$\S$9]{JoyceCS5} for a detailed proof.

At each self-intersection point of $N$ the tangent spaces are a pair of transverse 3-planes, which we can view as a pair of tranverse special Lagrangian 3-planes $P_1,P_2$ in $\C^3$.  Since 
we are in dimension 3, we know that there exists a (unique up to scale) special Lagrangian Lawlor neck $L$ asymptotic to $P_1\cup P_2$.  We can then glue $tL$ into $N$ near 
each intersection point to get a compact embedded  submanifold $S_t=N\# tL$ (if we glue in a Lawlor neck for every self-intersection point).  We can also arrange that $S_t$ is Lagrangian, 
i.e.~that it is a Lagrangian connect sum.

Now we want to perturb $S_t$ to be special Lagrangian.  Since $S_t$ is Lagrangian, by the deformation theory we can write any nearby submanifold as the graph of a 1-form $\alpha$, and 
this graph will be special Lagrangian if and only if (using the same notation as in our deformation theory discussion) $$P_t(\alpha)=(-*f_{\alpha}^*(\Imm\Upsilon),f_\alpha^*(\omega))=0.$$ 
 Since $S_t$ is Lagrangian but not special Lagrangian we have that
$$f_\alpha^*(\omega)=\d\alpha\quad\text{and}\quad -*f_{\alpha}^*(\Imm\Upsilon)=P_t(0)+\d_t^*\alpha+Q_t(\alpha,\nabla\alpha)$$
where $P_t(0)=-*\Im\Upsilon|_{S_t}$ and $\d_t^*=L_0P_t$, which is a perturbation of the usual $\d^*$ since we are no longer linearising at a point where $P_t(0)=0$.  By choosing $\alpha=\d f$, we then have to solve $$\Delta_tf=-P_t(0)-Q_t(\nabla f,\nabla^2 f)$$
where $\Delta_t$ is a perturbation of the Laplacian.

For simplicity, let us suppose that $\Delta_t$ is the Laplacian on $S_t$. 
The idea is to view our equation as a fixed point problem. We know that if we let $X^k=\{f\in C^{k,a}(N)\,:\,\int_Nf\vol_N=0\}$ then  
$\Delta_t: X^{k+2}\rightarrow X^k$
 is an isomorphism so it has an inverse
$G_t$.  We know by our elliptic regularity result that there exists a constant $C(\Delta_t)$ such that
$$\|f\|_{C^{k+2,a}} \leq C(\Delta_t)\|\Delta_tf\|_{C^{k,a}} \Leftrightarrow \|G_th\|_{C^{k+2,a}}\leq C(\Delta_t)\|h\|_{C^{k,a}}$$
for any $f\in X^{k+2}$, $h\in X^k$.  

We thus see that $P_t(f)=0$ for $f\in X^{k+2}$ if and only if
$$f = G_t(- P_t(0) - Q_t(f))= F_t(f).$$
The idea is now to show that $F_t$ is a contraction sufficiently near $0$ for all $t$ small enough.  Then it will have a (unique) fixed point near $0$, which 
will also be smooth because it satisfies $P_t(f)=0$ and hence defines a special Lagrangian as the graph of $\d f$ over $S_t$.

We know that $F_t:X^{k+2}\rightarrow X^{k+2}$ with
\begin{align*}
\|F_t(f_1) - F_t(f_2)\|_{C^{k+2,a}} &= \|G_t(Q_t(f_1) - Q_t(f_2))\|_{C^{k+2,a}}\leq C(\Delta_t)\|Q_t(f_1) - Q_t(f_2)\|_{C^{k,a}}.
\end{align*}
Since $Q_t$ and its first derivatives vanish at $0$ we know that
$$\|Q_t(f_1) - Q_t(f_2)\|_{C^{k,a}} \leq C(Q_t)\|f_1 - f_2\|_{C^{k+2,a}} (\|f_1\|_{C^{k+2,a}} + \|f_2\|_{C^{k+2,a}} ).$$
We deduce that
\[\|F_t(f_1)-F_t(f_2)\|_{C^{k+2,a}}\leq C(\Delta_t)C(Q_t)\|f_1-f_2\|_{C^{k+2,a}}(\|f_1\|_{C^{k+2,a}} + \|f_2\|_{C^{k+2,a}} )\]
and
\[\|F_t(0)\|_{C^{k+2,a}}=\|G_t(P_t(0))\|_{C^{k+2,a}}\leq C(\Delta_t)\|P_t(0)\|_{C^{k,a}}.\]
Hence, $F_t$ is a contraction on $\overline{B_{\epsilon_t}(0)} \subseteq X^{k+2}$ if we can choose $\epsilon_t$ so that 
$$2C(\Delta_t)\|P_t(0)\|_{C^{k,a}} \leq  \epsilon_t \leq \frac{1}{4C(\Delta_t)C(Q_t)}.$$
(This also proves Theorem \ref{SLdef.thm}, where we used the Implicit Function Theorem, by hand since there $P_t(0) = P(0) = 0$ so
we just need to take $\epsilon_t$ small enough.)
In other words, we need that
\begin{itemize}
\item $P_t(0)$ is small, so $S_t$ is ``close'' to being calibrated and is a good approximation to $P_t(f) = 0$;
\item $C(\Delta_t),C(Q_t)$, which are determined by the linear PDE and geometry of $N,L$ and $S_t$, are well-controlled as $t\rightarrow 0$. 
\end{itemize}
The statement of the theorem is then that there exists $t$ sufficiently small and $\epsilon_t$ so that the contraction mapping argument works.
  
This is a delicate balancing act since as $t\rightarrow 0$ parts of the manifold are collapsing, so the constants $C(\Delta_t),C(Q_t)$ above (which depend on $t$) can 
and typically do blow-up as $t\rightarrow 0$.  To control this, we need to understand the Laplacian on $N,L$ and $S_t$  and 
introduce ``weighted'' Banach spaces so that $tL$ gets rescaled to constant size
 (independent of $t$), and $S_t$ resembles the union of two manifolds with a cylindrical neck (as we described earlier).   It is also crucial to understand the relationship between the kernels and cokernels of the Laplacian 
on the \emph{non-compact} $N$ (with the intersection points removed), $L$ and compact $S_t$: here is where connectedness is important so that the kernel and cokernel of the Laplacian is 1-dimensional.
\end{proof}

\begin{remark}
In more challenging gluing problems it is not possible to show that the relevant map is a contraction, but rather one can instead 
appeal to an alternative  theorem (e.g.~Schauder fixed point theorem) to show that it still has a fixed point.
\end{remark}

\section{Associative and coassociative submanifolds}

We now want to introduce our calibrated geometry associated with $\GG_2$ holonomy.  The first key result is the following.

\begin{thm}\label{G2cab.thm}
Let $(M^7,\varphi)$ be a $\GG_2$ manifold (so $\varphi$ is a closed and coclosed positive 3-form).  Then $\varphi$ and $*\varphi$ are calibrations.
\end{thm}

\begin{proof}
Let $u,v,w$ be oriented orthonormal vectors in $\R^7$.  There exists an element $A$ of $\GG_2$ so that $Au=e_1$.  The subgroup of $\GG_2$ fixing $e_1$ is 
isomorphic to $\SU(3)$, and we know from the proof of Wirtinger's inequality (Theorem \ref{Wirtinger.thm}) there exists a (special) unitary transformation so that $v=e_2$ and 
$w=\cos\theta e_3+\sin\theta v$ for some $\theta$ and $v$ orthogonal to $e_1,e_2,e_3$. Since $\varphi(e_1,e_2,.)=\d x_3$, we see that 
$\varphi(u,v,w)=\cos\theta$.  Hence, since $\varphi$ is closed, $\varphi$ is a calibration and the calibrated planes are given by $A.\Span\{e_1,e_2,e_3\}$ for $A\in\GG_2$.

By Lemma \ref{starcab.lem}, $*\varphi$ is also a calibration.
\end{proof}

 Let us look at the calibrated planes and start with $\varphi$, which we take to be the following on $\R^7$:
 \[\varphi=\d x_{123}+\d x_{145}+\d x_{167} +\d x_{246} -\d x_{257} -\d x_{347} -\d x_{356},\]
 where we use the short-hand notation $\d x_{ij\ldots k}=\d x_i\wedge\d x_j\wedge\ldots\wedge \d x_k$.
 Hence, $*\varphi$ on $\R^7$ is given by:
 \[*\varphi=\d x_{4567}+\d x_{2367} +\d x_{2345} +\d x_{1357} -\d x_{1346} -\d x_{1256} -\d x_{1247}.\]

If $u,v,w$ are unit vectors in $\R^7\cong \Imm\O$ (the imaginary octonions), then $\varphi(u,v,w)=\langle u\times v,w\rangle=1$ if and only if 
$w=u\times v$, so $P=\Span\{u,v,w\}$ is a copy of $\Imm\H$ in $\Imm\O$; in other words, $\Span\{1,u,v,w\}$ is an associative subalgebra of $\O$.  
Moreover, suppose we define a vector-valued 3-form $\chi$ on $\R^7$ by
\begin{equation*}
\chi(u,v,w)=[u,v,w]=u(vw)-(uv)w,
\end{equation*} known 
as the \emph{associator}.  Then we observe the following.

\begin{lem}
A 3-plane $P$ in $\R^7$ satisfies $\chi|_P\equiv 0$ if and only if $P$ admits an orientation so that it is calibrated by $\varphi$.
\end{lem}

\begin{proof}
  Since the associator is clearly invariant under $\GG_2$ we can put any plane $P$ in standard position using $\GG_2$, i.e.~as in the proof of Theorem \ref{G2cab.thm}, we can write
$P=\Span\{e_1,e_2,\cos\theta e_3+\sin\theta v\}$ for some $v$ orthogonal to $e_1,e_2,e_3$.  
We can calculate that $[e_1,e_2,e_3]=0$ whereas $[e_1,e_2,v]\neq 0$ for any $v$ orthogonal to $e_1,e_2,e_3$.  Moreover, $P$ is calibrated by
$\varphi$ if and only if $\theta=0$.
We thus see that $P$ is calibrated by $\varphi$ (up to a choice a orientation) if and only if $\chi|_P\equiv 0$. 
\end{proof}

Hence we call the $\varphi$-calibrated planes \emph{associative}.   In general on a $\GG_2$ manifold we can define a 3-form $\chi$ with
 values in $TM$ using the pointwise formula above.
 
 For $*\varphi$ we see that $*\varphi|_P=\vol_P$ for a plane $P$ if and only if $\varphi|_{P^{\perp}}=\vol_{P^{\perp}}$.  Hence the planes calibrated by 
$*\varphi$ are the orthogonal complements of the associative planes, so we call them \emph{coassociative}.  
 We have a similar alternative characterisation for 4-planes calibrated by $*\varphi$.
 
\begin{lem}
A 4-plane $P$ in $\R^7$ satisfies $\varphi|_P\equiv 0$ if and only if $P$ admits an orientation so that it is calibrated by $*\varphi$.
\end{lem}

\begin{proof}
We know that given a 4-plane $P$ we can choose coordinates such that  $P^{\perp}=\Span\{e_1,e_2,\cos\theta e_3+\sin\theta (a_4 e_4+a_5e_5+a_6e_6+a_7e_7)\}$  
where $\sum_ja_j^2=1$.
Then 
\begin{align*}
P=\Span\{&-\sin\theta e_3+\cos\theta (a_je_j),a_5e_4-a_4e_5+a_7e_6-a_6e_7,\\
&a_6e_4-a_7e_5-a_4e_6+a_5e_7,a_7e_4+a_6e_5-a_5e_6 -a_4e_7\}.
\end{align*}
We can then see directly that $*\varphi|_P=\cos\theta$.  We also have $\varphi(e_i,e_j,e_k)=0$ for $i,j,k\in\{4,5,6,7\}$ and $e_3\lrcorner\varphi
=-\d x_{47}-\d x_{56}$, so that $\varphi(-\sin\theta e_3+\cos \theta(a_je_j),v,w)$ is a non-zero multiple of $\sin\theta$ for some $v,w\in P$.  
Hence $\varphi|_P=0$ if and only if $\theta=0$, which is if and only if $P$ is calibrated by $*\varphi$ (again up to a choice of orientation).
\end{proof}

We thus can define our calibrated submanifolds.

\begin{dfn}
Submanifolds in $(M^7,\varphi)$ calibrated by $\varphi$ are called \emph{associative} 3-folds.  Moreover, $N$ is associative if and only if $\chi|_N\equiv 0$ (up to a choice of orientation).

Submanifolds in $(M^7,\varphi)$ calibrated by $*\varphi$ are called 
\emph{coassociative} 4-folds. Moreover, $N$ is coassociative if and only if $\varphi|_N\equiv 0$ (up to a choice of orientation).
\end{dfn}

It is instructive to see the form that the associative or coassociative condition takes by studying associative or coassociative graphs in $\R^7$: see \cite{HL} for details.

A simple way to get associative and coassociative submanifolds is by using known geometries.

\begin{prop}
Let $x_1,\ldots,x_7$ be coordinates on $\R^7$ and let $z_j=x_{2j}+ix_{2j+1}$ be coordinates on $\C^3$ so that $\R^7=\R\times\C^3$.  
\begin{itemize}
\item[(a)] $N=\R\times S\subseteq \R\times\C^3$ is associative or coassociative if and only if $S$ is a complex curve or a special Lagrangian 3-fold with 
phase $-i$, respectively.
\item[(b)] $N\subseteq\{0\}\times\C^3$ is associative or coassociative if and only if $N$ is a special Lagrangian 3-fold or a complex surface, respectively.
\end{itemize}
\end{prop}

\begin{proof}
Recall the K\"ahler form $\omega$ and holomorphic volume form $\Upsilon$ on $\C^3$.  We can write 
$$\varphi=\d x_1\wedge\omega+\Ree\Upsilon\quad\text{and}\quad *\varphi=\frac{1}{2}\omega^2-\d x_1\wedge\Imm\Upsilon.$$
For associatives, we see that $\varphi|_{\R\times S}=\d x_1\wedge\vol_S$ if and only if $\omega|_S=\vol_S$ and $\varphi|_N=\Ree\Upsilon|_N$ for $N\subseteq\C^3$.
For coassociatives, we see that $*\varphi|_{\R\times S}=\d x_1\wedge\vol_S$ if and only if $-\Imm\Upsilon|_S=\vol_S$ and $*\varphi|_N=\frac{1}{2}\omega^2|_N$ 
for $N\subseteq\C^3$.  The results quickly follow.
\end{proof}

We can also produce examples in $\GG_2$ manifolds with an isometric involution.

\begin{prop}
Let $(M,\varphi)$ be a $\GG_2$ manifold with an isometric involution $\sigma\neq \id$ such that $\sigma^*\varphi=\varphi$ or $\sigma^*\varphi=-\varphi$.  Then 
$\text{Fix}(\sigma)$ is an associative or coassociative submanifold in $M$ respectively, if it is non-empty.
\end{prop}

We also have explicit examples of associatives and coassociatives.

\begin{ex}
The first explicit examples of associatives in $\R^7$ not arising from other geometries are given in \cite{Lotayevol} from symmetry and evolution equation considerations.
%, which are asymptotic (sometimes in a weak 
%sense) to cones (or a double cover of a cone) at infinity.

The first explicit non-trivial examples of coassociatives in $\R^7$ are given in \cite{HL}.  There are two dilation families: one which has one end 
asymptotic to a cone $C$ on a non-round $\mathcal{S}^3$, and one which has two ends 
asymptotic to $C\cup\R^4$.  The cone $C$ was discovered earlier by Lawson--Osserman \cite{LO} and was the first example of a volume-minimizing submanifold which 
is not smooth (it is Lipschitz but not $C^1$).
\end{ex}

\begin{ex}
In the Bryant--Salamon holonomy $\GG_2$ metric on the spinor bundle of $\mathcal{S}^3$ \cite{BS}, the base $\mathcal{S}^3$ is associative.
In the Bryant--Salamon holonomy $\GG_2$ metrics on $\Lambda^2_+T^*\mathcal{S}^4$ and $\Lambda^2_+T^*\C\P^2$ \cite{BS}, the bases $\mathcal{S}^4$ and $\C\P^2$ are coassociative. 
 %Moreover,  Karigiannis--Min-Oo build associative line bundles over any minimal surface and coassociative 2-plane bundles over any superminimal surface.
\end{ex}

We now want to understand deformations of associatives and coassociatives, from which perturbation or gluing results will follow.  We begin with associatives.

%\begin{ex}
%If $\SS$ is the spin bundle, we have a spin connection $\nabla:C^{\infty}(\SS)\rightarrow C^{\infty}(T^*M\otimes \SS)$ (lift of the Levi-Civita connection) and 
%we have Clifford multiplication $m:C^{\infty}(T^*M\otimes \SS)\rightarrow C^{\infty}(\SS)$ given by $m(\xi,v)=\xi\cdot v$.  Hence we have a composition
%$\Dslash=m\circ\nabla:C^{\infty}(\SS)\rightarrow C^{\infty}(\SS)$, which is a first order linear differential operator called the \emph{Dirac operator}.  
%Locally it is given by $\Dslash v= \sum_i e_i\cdot\nabla_{e_i} v$, so we have that
%$$\sigma_{\Dslash}(\xi,v)=\xi\cdot v.$$
%Hence $\Dslash$ is elliptic.  Moreover $\Dslash$ is self-adjoint.  
%More generally, given any other vector bundle $E$ with a connection on it we can define a (twisted) Dirac operator 
%$\Dslash:C^{\infty}(\SS\otimes E)\rightarrow C^{\infty}(\SS\otimes E)$, using $\nabla(v\otimes w)=\nabla v\otimes w+v\otimes \nabla w$ and 
%$m(\xi\otimes v\otimes w)=(\xi\cdot v)\otimes w$. In general a Dirac operator $\Dslash$ is a linear first order elliptic operator such that 
%the symbol of $\Dslash^2$ is the symbol of the Laplacian.
%\end{ex}

Notice that if $P$ is an associative plane, $u\in P$ and $v\in P^{\perp}$ then $u\times v\in P^{\perp}$ since 
$\varphi(w,u,v)=g(w,u\times v)=g(v,w\times u)=0$ for all $w\in P$ since $w\times u\in P$. Thus, if $N$ is associative, cross product 
 gives a (Clifford) multiplication $m:C^{\infty}(T^*N\otimes\nu(N))\rightarrow C^{\infty}(\nu(N))$ (viewing tangent vectors as cotangent vectors via the metric). Hence, using the normal connection $\nabla^{\perp}:C^{\infty}(\nu(N))
\rightarrow C^{\infty}(T^*N\otimes\nu(N))$ on 
$\nu(N)$ we get a linear operator $$\Dslash=m\circ\nabla^{\perp}:C^{\infty}(\nu(N))\rightarrow C^{\infty}(\nu(N)).$$  We call $\Dslash$ the \emph{Dirac operator}.
 We see that its principal symbol is given by $\sigma_{\Dslash}(x,\xi)v=i\xi\times v$, so $\Dslash$ is elliptic, and we also have that $\Dslash^*=\Dslash$.

\begin{remark}
Since a 3-manifold is always spin, we have a spinor bundle $\SS$ on $N$, a connection $\nabla:C^{\infty}(\SS)\rightarrow C^{\infty}(T^*M\otimes \SS)$ (a lift of the Levi-Civita connection) and 
we have Clifford multiplication $m:C^{\infty}(T^*M\otimes \SS)\rightarrow C^{\infty}(\SS)$ given by $m(\xi,v)=\xi\cdot v$.  Hence we have a composition
$\Dslash=m\circ\nabla:C^{\infty}(\SS)\rightarrow C^{\infty}(\SS)$, which is a first order linear differential operator called the Dirac operator.   
Locally it is given by $\Dslash v= \sum_i e_i\cdot\nabla_{e_i} v$, so we have that
$\sigma_{\Dslash}(\xi,v)=i\xi\cdot v$.  
Hence $\Dslash$ is elliptic.  Moreover $\Dslash$ is self-adjoint. 

In fact, it is possible (see e.g.~\cite{McLean}) to see that the complexified normal bundle 
$\nu(N)\otimes\C=\SS\otimes V$ for a $\C^2$-bundle $V$ over $N$, so that the Dirac operator on $\nu(N)$ is just a ``twist'' of the usual Dirac operator on $\SS$.
\end{remark}

Consider a compact associative $N$.  We want to describe the associative deformations of $N$, just as in the case of special Lagrangians above.  To be consistent with that previous setting, we will now use $P$ to denote a nonlinear deformation map: we trust that this will not cause confusion given the context.

 We know that $\exp_v(N)=N_v$, which is the graph of $v$, is associative for a normal vector field $v$ if and only if 
$*\exp_v^*(\chi)\in C^{\infty}(TM|_N)$
is $0$.  In fact, it turns out that $P(v)=*\exp_v^*(\chi)\in C^{\infty}(\nu(N))$ since $N$ is associative and 
$$L_0P(v)=*\d(v\lrcorner\chi)=\Dslash v.$$
Here $L_0P$ is not typically surjective so we cannot apply our Implicit Function Theorem, except when $\Ker\Dslash=\Ker \Dslash^*=\{0\}$.  However, we can still say something in these circumstances, for which we make a small digression to a more general situation.

Suppose $X,Y$ are Banach spaces.  Let $U\subseteq X$ be an open set with $0\in U$ and let $P:U\rightarrow Y$ be a smooth map with $P(0)=0$ such that $L_0P:X\rightarrow Y$
is Fredholm.  

Let $\mathcal{I}=\Ker L_0P$ and let $\mathcal{O}$ be such that $Y=L_0P(X)\oplus\mathcal{O}$, which exists and is finite-dimensional by the assumption that
$L_0P$ is Fredholm.  We then let $Z=X\oplus\mathcal{O}$ and define $F:U\oplus\mathcal{O}\rightarrow Y$ by $F(u,y)=P(u)+y$.  We see that 
$L_0F:X\oplus\mathcal{O}\rightarrow Y=L_0P(X)\oplus\mathcal{O}$ is given by $L_0F(x,y)=L_0P(x)+y$ which is surjective and $L_0F(x,y)=0$ if and only if 
$L_0P(x)=0$ and $y=0$, thus $\Ker L_0F=\Ker L_0P\times\{0\}$.  

There exists $W\subseteq X$ such that $\Ker L_0P\oplus W=X$. 
Applying the Implicit Function Theorem, there exist  open sets $U_1\subseteq \Ker L_0P$ containing $0$, $U_2\subseteq W$ containing $0$ and 
$U_3\subseteq\mathcal{O}$ containing $0$ and smooth maps $G_2:U_1\rightarrow U_2$, $G_3:U_1\rightarrow U_3$ such that 
$$F^{-1}(0)\cap U_1\times U_2\times U_3=\{(u,G_2(u),G_3(u))\,:\,u\in U_1\}.$$
We also know that $P(x)=0$ if and only if $F(x,y)=0$ and $y=0$.  Hence
$$P^{-1}(0)\cap U_1\times U_2=\{(u,G_2(u))\,:\,u\in G_3^{-1}(0)\}.$$
Let $\mathcal{U}=U_1$ and define $\pi:\mathcal{U}\rightarrow\mathcal{O}$ by $\pi(u)=G_3(u)$.  Then $P^{-1}(0)\cap U_1\times U_2$ is a graph over $\pi^{-1}(0)$, 
and hence $P^{-1}(0)$ is locally homeomorphic to $\pi^{-1}(0)$.

Sard's Theorem says that generically $\pi^{-1}(y)$ is a smooth manifold of dimension $\dim\mathcal{I}-\dim\mathcal{O}=\dim\Ker L_0P-\dim\text{Coker}\,L_0P$, 
which is the index of $L_0P$.  Hence, the expected dimension of $P^{-1}(0)$ is the index of $L_0P$.
  
In the associative setting we have that the linearisation is $\Dslash$, which is elliptic and thus Fredholm, and we know that $\text{index}\,\Dslash=\dim\Ker\Dslash-\dim\Ker\Dslash^*=0$.  We deduce the following \cite{McLean}.

\begin{thm}\label{assocdef.thm}
The expected dimension of the moduli space of deformations of a compact associative 3-fold $N$ in a $\GG_2$ manifold is $0$ and infinitesimal deformations of $N$ are given 
by the kernel of $\Dslash$ on $\nu(N)$.  Moreover, if $\Ker\Dslash=\{0\}$ then $N$ is rigid.
\end{thm}

\begin{remark}
The dimension of the kernel of $\Dslash$ typically depends on the metric on $N$ rather than just the topology, so it is usually difficult to determine.  However, there are some cases where one can ensure the moduli space is smooth cf.~\cite{Gayet}.
\end{remark}

%\begin{proof}
%For $w\in\Ker\Dslash=\text{Coker}\,\Dslash$ define $P(v,w)=F(v)+w$ Then $L_0P(v,w)=\Dslash v+w$ which is now surjective.
%Therefore $P^{-1}(0)$ is smooth manifold of dimension equal to the kernel of $L_0P$, which is $\Ker\Dslash$ since $\Dslash v+w=0$ if and only if $\Dslash v=w=0$.  
%Hence the moduli space, which is $F^{-1}(0)$ is $\pi^{-1}(0)$ where $\pi:P^{-1}(0)\rightarrow\Ker\Dslash$ is the projection $\pi(v,w)=w$.  The result follows as
% the expected dimension is $\dim P^{-1}(0)-\dim\Ker\Dslash=0$.
%\end{proof}

\begin{ex}
For the associative $N=\mathcal{S}^3$ in $\SS(\mathcal{S}^3)$, $\nu(N)=\SS(\mathcal{S}^3)$ so $\Dslash$ is just the usual Dirac operator.  
A theorem of Lichnerowicz states that $\Ker\Dslash=\{0\}$ as $\mathcal{S}^3$ has positive scalar curvature so $N$ is rigid.
%The Weitzenb\"ock formula states that $\Dslash^2=\nabla^*\nabla+R$ where $R$ is the scalar curvature.  Since $R>0$ in this case we see that if 
%$\Dslash v=0$ then $0=\langle v,\Dslash^2 v\rangle=\|\nabla v\|^2+R\|v\|^2$ so $v=0$.  Hence $\Ker\Dslash=\{0\}$ so $N$ is rigid.
\end{ex}

\begin{ex}
Corti--Haskins--Nordstr\"om--Pacini construct rigid associative $\mathcal{S}^1\times\mathcal{S}^2$s in compact holonomy $\GG_2$ twisted connected sums \cite{CHNP}.
\end{ex}

For coassociatives, the deformation theory is much better behaved, like for special Lagrangians \cite{McLean}.

\begin{thm}
Let $N$ be a compact coassociative in a $\GG_2$ manifold (or just a 7-manifold with closed $\GG_2$ structure).  The moduli space of deformations of $N$ is a smooth manifold of dimension $b^2_+(N)$.
\end{thm}

\begin{proof}
Since $N$ is coassociative the map $v\mapsto v\lrcorner\varphi=\alpha_v$ defines an isomorphism from $\nu(N)$ to a rank 3 vector bundle on $N$, which is 
$\Lambda^2_+T^*N$, the 2-forms on $N$ which are self-dual (so $*\alpha=\alpha$).  We can therefore view nearby submanifolds to $N$ as graphs of self-dual 2-forms.

We know that $N_v=\exp_v(N)$ is coassociative if and only if $\exp_v^*(\varphi)=0$.  We see that
$$\frac{\d}{\d t}\exp_{tv}^*(\varphi)|_{t=0}=\mathcal{L}_v\varphi=\d(v\lrcorner\varphi)=\d\alpha_v.$$
Hence nearby coassociatives $N'$ to $N$ are given by the zeros of $P(\alpha)=\d\alpha+Q(\alpha,\nabla\alpha)$.  Moreover, since $\varphi=0$ on $N$, $[\varphi]=0$ on $N'$
and hence $P:C^{\infty}(\Lambda^2_+T^*N)\rightarrow \d (C^{\infty}(\Lambda^2T^*N))$.  

Here $P$ is not elliptic, but $L_0P=\d$ has finite-dimensional kernel, the closed self-dual 2-forms,  since $\d\alpha=0$ implies that $\d^*\alpha
=-*\d*\alpha=0$ so $\alpha$ is harmonic. Moreover, $L_0P$ has injective symbol so it is overdetermined elliptic, which means that elliptic regularity still holds.  
Another way to deal with this is to consider $F(\alpha,\beta)=P(\alpha)+\d^*\beta$ for $\beta$ a 4-form.  Now $F^{-1}(0)$ is the disjoint union of $P^{-1}(0)$  
and  multiples of the volume form,  as exact  and coexact 
forms are orthogonal.  Moreover, $L_0F(\alpha,\beta)=\d\alpha+\d^*\beta$ is now elliptic.   Overall, we can apply our standard Implicit Function Theorem 
if we know that
$$\d(C^{k+1,a}(\Lambda^2_+T^*N))=\d(C^{k+1,a}(\Lambda^2T^*N)).$$
This is true because by Hodge theory if $\alpha$ is a 2-form, we can write $\alpha=\d^*\beta+\gamma$ for a 3-form $\beta$ and a closed form $\gamma$, 
so $\d\alpha=\d\d^*\beta=\d(\d^*\beta+*\d^*\beta)$ and $\d^*\beta+*\d^*\beta$ is self-dual.
\end{proof}

\begin{ex}
The $\mathcal{S}^4$ and $\C\P^2$ in the Bryant--Salamon metrics on $\Lambda^2_+T^*\mathcal{S}^4$ and $\Lambda^2_+T^*\C\P^2$ have $b^2_+=0$ and so are rigid.
\end{ex}

For a K3 surface and $T^4$ we have $b^2_+=3$ and $\Lambda^2_+$ is trivial, so we can hope to find coassociative K3 and $T^4$ fibrations of compact $\GG_2$ manifolds.  There is 
a programme \cite{Kovalev} for constructing a coassociative K3 fibration (with singular fibres).  Towards completing this programme,  the first 
examples of compact coassociative 4-folds with conical singularities in compact holonomy $\GG_2$ twisted connected sums were constructed in \cite{Lotaystab}.

Again, we have a similar isometric embedding result for coassociative 4-folds, motivated by the deformation theory result \cite{Bryant.embed}.

\begin{thm} 
Any compact oriented real analytic Riemannian 4-manifold whose bundle of self-dual 2-forms is trivial can be isometrically embedded in a $\GG_2$ manifold 
as the fixed points of an isometric involution.
\end{thm}

\begin{remark}
The deformation theory results for compact associative and coassociative submanifolds have been extended to certain non-compact, singular and boundary settings, for example in \cite{GayetWitt, JoySalur, KovalevLotay, LotayCScoass, LotayACcoass, LotayACass}.
\end{remark}

%\begin{center}
%{\large Lecture 5: Cayleys, the angle theorem}
%\end{center}

\section{Cayley submanifolds}

We now discuss our final class of calibrated submanifolds.

\begin{thm}
On a $\Spin(7)$ manifold $(M^8,\Phi)$ (so $\Phi$ is a closed admissible form), $\Phi$ is a calibration.
\end{thm}

\begin{proof}
Let $P$ be a plane in $\R^8\cong\C^4$.
Since $\SU(4)\subseteq\Spin(7)$, by the proof of Wirtinger's inequality (Theorem \ref{Wirtinger.thm}), we can choose $A\in\Spin(7)$ so that $A(P)$ is spanned by 
$$\{e_1,\cos\theta_1 e_2+\sin\theta_1e_3,e_5,\cos\theta_2 e_6+\sin\theta_2e_7\}.$$
We take the standard $\Spin(7)$ form $\Phi$ on $\R^8$ to be:
\begin{align*}\Phi=&\;\d x_{1234}+\d x_{1256}+\d x_{1278} +\d x_{1357} -\d x_{1367} -\d x_{1458} -\d x_{1467}\\
&+\d x_{5678}+\d x_{3478} +\d x_{3456} +\d x_{2468} -\d x_{2457} -\d x_{2367} -\d x_{2358},
\end{align*}
again using the notation  $\d x_{ij\ldots k}=\d x_i\wedge\d x_j\wedge\ldots\wedge \d x_k$.
Then $\Phi|_P=(\cos\theta_1\cos\theta_2+\sin\theta_1\sin\theta_2)\vol_P=\cos(\theta_1-\theta_2)\vol_P$.  Hence $\Phi$ is a calibration as it is closed.  
\end{proof}

We can thus define our calibrated submanifolds in $\Spin(7)$ manifolds.

\begin{dfn}
Submanifolds in $(M^8,\Phi)$ calibrated by $\Phi$ are called Cayley 4-folds. 
\end{dfn}

\begin{remark}
The name Cayley submanifolds is because of the relation between the submanifolds and the octonions or Cayley numbers $\O$.
\end{remark}

We can relate Cayley submanifolds to all of the other calibrated geometries we have seen.

\begin{prop}
\begin{itemize}
\item[(a)] Complex surfaces and special Lagrangian 4-folds in $\C^4$ are Cayley in $\R^8=\C^4$.
\item[(b)] Write $\R^8=\R\times\R^7$.  Then $\R\times S$ is Cayley if and only if $S$ is associative in $\R^7$ and $N\subseteq \R^7$ is 
Cayley in $\R^8$ if and only if $N$ is coassociative in $\R^7$.
\end{itemize}
\end{prop}

\begin{proof}
Recall the K\"ahler form $\omega$ and holomorphic volume form $\Upsilon$ on $\C^4$ and the $\GG_2$ 3-form $\varphi$ on $\R^7$.

Part (a) is immediate from the formula $\Phi=\frac{1}{2}\omega^2+\Ree\Upsilon$, since complex surfaces are calibrated by $\frac{1}{2}\omega^2$, special Lagrangians are calibrated by $\Ree\Upsilon$, $\Upsilon$ vanishes on complex surfaces and $\omega$ vanishes on special 
Lagrangians.

Part (b) follows immediately from the formula $\Phi=\d x_1\wedge\varphi+*\varphi$.
\end{proof}

We can also use an isometric involution to construct Cayley submanifolds as in our previous calibrated geometries.

\begin{prop}
Let $(M,\Phi)$ be a $\Spin(7)$ manifold and let $\sigma\neq\id$ be an isometric involution with $\sigma^*\Phi=\Phi$.  Then $\text{Fix}(\sigma)$ is Cayley 
submanifold, if it is non-empty.
\end{prop}

\begin{ex}
The first interesting explicit examples of Cayleys in $\R^8$ not arising from other geometries were given in \cite{Lotay2R} and are asymptotic to cones.
\end{ex}

\begin{ex}
The base $\mathcal{S}^4$ in the Bryant--Salamon holonomy $\Spin(7)$ metric on $\SS_+(\mathcal{S}^4)$ \cite{BS} is Cayley.  
%
%One can also build Cayley 2-plane bundles over minimal surfaces in $\mathcal{S}^4$.
\end{ex}

We now discuss deformations of a compact Cayley $N$, for which we need  some discussion of algebra related to $\Spin(7)$.
Since $\Lambda^2(\R^8)^*$ is 28-dimensional and the 21-dimensional Lie algebra of $\Spin(7)$ sits inside the space of $2$-forms, we must have a 
distinguished $7$-dimensional subspace $\Lambda^2_7$ of 2-forms on $\R^8$.  So what is this subspace?
 Let $u,v\in\R^8$.  Then we can construct a 2-form $u\wedge v$, viewing $u,v$ as cotangent vectors.  We can also construct a 2-form from $u,v$ by 
 considering $\Phi(u,v,.,.)$.  It is then true that 
 $$\Lambda^2_7=\{u\wedge v+\Phi(u,v,.,.)\,:\,u,v\in\R^8\}.$$
 When $P$ is a Cayley plane and $u,v\in P$ are orthogonal we see that $\Phi(u,v,.,.)=*_P(u\wedge v)$ so that $u\wedge v+\Phi(u,v,.,.)$ is self-dual on $P$.  Since 
 $\Lambda^2_+P^*$ is 3-dimensional, we see that there must be a 4-dimensional space $E$ of 2-forms on $P$ such that $\Lambda^2_7|_P=\Lambda^2_+P^*\oplus E$.
Moreover, if $u\in P$ and $v\in P^{\perp}$ then $m(u,v)=u\wedge v+\Phi(u,v,.,.)\in E$ and the map $m:P\times P^{\perp}\rightarrow E$ is surjective.

Now let us move to a Cayley submanifold $N$ in a $\Spin(7)$ manifold $(M,\Phi)$.  On $M$ we have a rank 7 bundle $\Lambda^2_7$ of 2-forms and we have that
$\Lambda^2_7|_N=\Lambda^2_+T^*N\oplus E$ for some rank 4 bundle $E$ over $N$.  The map $m$ above defines a (Clifford) multiplication 
$m:C^{\infty}(T^*N\otimes \nu(N))\rightarrow C^{\infty}(E)$ (viewing tangent vectors as cotangent vectors via the metric), and thus using the normal connection 
$\nabla^{\perp}:C^{\infty}(\nu(N))\rightarrow C^{\infty}(T^*N\otimes \nu(N))$ we get a linear first order differential operator
$$\Dslash_+=m\circ\nabla^{\perp}:C^{\infty}(\nu(N))\rightarrow C^{\infty}(E).$$
Again this an elliptic operator called the \emph{positive Dirac operator}, but it is not self-adjoint: its adjoint is the negative Dirac operator from $E$
to $\nu(N)$.

\begin{remark} If $N$ is spin, the spinor bundle $\SS$ splits as $\SS_+\oplus\SS_-$, and the Dirac operator $\Dslash$ splits into 
$\Dslash_{\pm}$ from $\SS_{\pm}$ to $\SS_{\mp}$ so that $\Dslash(v_+,v_-)=(\Dslash_-v_-,\Dslash_+v_+)$.  Hence $\Dslash^*=\Dslash$ says that 
$\Dslash_{\pm}^*=\Dslash_{\mp}$.

 It turns out (see, for example, \cite{McLean}) that there exists 
a $\C^2$-bundle $V$ on $N$ so that $\nu(N)\otimes\C=\SS_+\otimes V$, $E\otimes\C=\SS_-\otimes V$ and $\Dslash_+$ on $\nu(N)$ is a ``twist'' of the 
usual positive Dirac operator. However, not every 4-manifold is spin, so we cannot always make this identification.
\end{remark}

On $\O$ there exists 
a 4-fold cross product, whose real part gives $\Phi$ and whose imaginary part we call $\tau$.  Perhaps unsurprisingly, we have the following result, which we will leave as an exercise for the reader.

\begin{lem} A 4-plane $P$ in $\R^8$ satisfies $\tau|_P\equiv 0$ if and only if it admits an orientation so that it is calibrated by $\Phi$.
\end{lem}
 
 We can extend $\tau$ to a $\Spin(7)$ manifold, except that we need a rank 7 vector bundle on $M$ in which $\tau$ takes values: 
we have one, namely $\Lambda^2_7$. So we have the following alternative characterisation of Cayley 4-folds.

\begin{lem}
A submanifold $N$ in a $\Spin(7)$ manifold is Cayley (up to a choice of orientation) if and only if $\tau\in C^{\infty}(\Lambda^4T^*M;\Lambda^2_7)$ vanishes on $N$. 
\end{lem}

Now suppose that $N$ is a compact Cayley 4-fold.
Then  the zeros of the equation $F(v)=*\exp_v^*(\tau)$ for $v\in C^{\infty}(\nu(N))$ define Cayley deformations (as the graph of $v$).
 We know that  $F$ takes values in $\Lambda^2_7|_N=\Lambda^2_+T^*N\oplus E$ and it turns out that
$$L_0F(v)=*\d(v\lrcorner\tau)=\Dslash_+v$$
since $N$ is Cayley.  So, we potentially have a problem because $F$ does not necessarily take values only in $E$ (and in general it will not just take values 
in $E$).  However, the Cayley condition on $N$ means that $F(v)=0$ if and only $P(v)=\pi_EF(v)=0$, where 
$\pi_E$ is the projection onto $E$ (again, we are using $P$ to denote the nonlinear deformation map as in our previous discussion, and we expect it will not cause confusion given the context).  Then the operator $P:C^{\infty}(\nu(N))\rightarrow C^{\infty}(E)$ and $L_0P=\Dslash_+$ is elliptic.

Again, we cannot say that $L_0P$ is surjective, so we have the following using the same argument as in the lead up to Theorem \ref{assocdef.thm}, cf.~\cite{McLean}.

\begin{thm}\label{Caydef.thm}
The expected dimension of the moduli space of deformations of a compact Cayley 4-fold $N$ in a $\Spin(7)$ manifold is
$\text{\emph{ind}}\,\Dslash_+=\dim\Ker\Dslash_+-\dim\Ker\Dslash_+^*$
 with infinitesimal deformations given by $\Ker\Dslash_+$ on $\nu(N)$. Moreover,
$$\text{\emph{ind}}\,\Dslash_+=\frac{1}{2}\sigma(N)+\frac{1}{2}\chi(N)-[N].[N],$$
where $\sigma(N)=b^2_+(N)-b^2_-(N)$ (the signature of $N$), $\chi(N)=2b^0(N)-2b^1(N)+b^2(N)$ (the Euler characteristic of $N$) and 
$[N].[N]$ is the self-intersection of $N$, which is the Euler number of $\nu(N)$.
\end{thm}

\begin{ex}
For the Cayley $N=\mathcal{S}^4$ in $\SS_+(\mathcal{S}^4)$, $\nu(N)=\SS_+(\mathcal{S}^4)$ and $\Dslash_+$ is the usual positive Dirac operator.  Again, since 
$N$ has positive scalar curvature, we see that $\Ker\Dslash_{\pm}=\{0\}$ so $N$ is rigid.
\end{ex}

\begin{remark}
Theorem \ref{Caydef.thm} has been extended to various other non-compact, singular and boundary settings, for example in \cite{Moore.CS, Ohst1, Ohst2}.
\end{remark}

\section{The angle theorem}\label{s:angle}

To conclude these notes, we now discuss a very natural and elementary problem in Euclidean geometry where calibrations play a major, and perhaps unexpected, role.

If one takes two lines in $\R^2$ intersecting transversely, then their union is never length-minimizing.  A natural question to ask is: does 
this persist in higher dimensions?  In other words, when is the union of two transversely intersecting $n$-planes in $\R^{2n}$ volume-minimizing?  Two such planes 
are determined by the $n$ angles between them as follows.

\begin{lem}\label{angles.lem}
Let $P,Q$ be oriented $n$-planes in $\R^{2n}$.  There exists an orthonormal basis $e_1,\ldots,e_{2n}$ for $\R^{2n}$ such that 
$P=\Span\{e_1,\ldots,e_n\}$ and
$$Q=\Span\{\cos\theta_1 e_1+\sin\theta_1e_{n+1},\ldots,\cos\theta_n e_n+\sin\theta_n e_{2n}\}$$
where $0\leq\theta_1\leq\ldots\leq\theta_{n-1}\leq\frac{\pi}{2}$ and $\theta_{n-1}\leq\theta_n\leq\pi-\theta_{n-1}$.  These angles are 
called the \emph{characterising angles} of $P,Q$.
\end{lem}

\begin{proof}
The proof is very similar to the argument in the proof of Wirtinger's inequality (Theorem \ref{Wirtinger.thm}).  Choose unit $e_1\in P$ and maximise $\langle e_1,u_1\rangle$ for $u_1\in Q$, and let 
$e_{n+1}\in P^{\perp}$ be defined by $u_1=\cos\theta_1e_1+\sin\theta_1e_{n+1}$.  Now choose $e_2\in P\cap e_1^{\perp}$ and maximise 
$\langle e_2,u_2\rangle$ for $u_2\in Q\cap u_1^{\perp}$, then proceed by induction.  
\end{proof}

If the characterising angles of $P,Q$ are $\theta_1,\ldots,\theta_n$, then the characterising angles of $P,-Q$ are $\psi_1,\ldots,\psi_n$ where 
$\psi_j=\theta_j$ for $j=1,\ldots,n-1$ and $\psi_n=\pi-\theta_n$.

The idea of the following theorem is that the union of $P\cup Q$ is area-minimizing if $P,-Q$ are not too close together \cite{Lawlor}.

\begin{thm}[Angle Theorem]
Let $P,Q$ be oriented transverse $n$-planes in $\R^{2n}$ and let $\psi_1,\ldots,\psi_n$ be the characterising angles between $P,-Q$.  Then $P\cup Q$ is volume-minimizing if and only if 
$\psi_1+\ldots+\psi_n\geq \pi$.
\end{thm}

Notice that this criteria is impossible to fulfill in 1 dimension.

\begin{proof}
We will sketch the proof which involves calibrations in a fundamental way in both directions.  For details, look at \cite{Harvey}.  

First if $P\cup Q$ does not satisfy the angle condition, we can 
choose coordinates by Lemma \ref{angles.lem} so that $P=P(-\frac{\psi}{2})$ and $-Q=P(\frac{\psi}{2})$ where $\psi=(\psi_1,\ldots,\psi_n)$ and $P(\psi)=\{(e^{i\psi_1,}x_1,\ldots,e^{i\psi_n}x_n)\,:\,(x_1,\ldots,x_n)\in\R^n\}$ as given earlier.  
We know that we have a special Lagrangian Lawlor neck $N$ asymptotic to $P(-\frac{\psi'}{2})\cup P(\frac{\psi'}{2})$ for any $\psi'$ where $\sum_{i=1}^n\psi'_i=\pi$.  
The claim is then that since $\sum \psi_i<\pi$ we can find $\psi'$ so that $\sum\psi'_i=\pi$ and $N\cap P(\pm\frac{\psi'}{2})$ is non-empty (in fact, an ellipsoid).  
This is actually a way to characterise $N$.  Hence since $N$ is calibrated by $\Imm\Upsilon$ and $\Imm\Upsilon|_{P\cup Q}<\vol_{P\cup Q}$, 
$P\cup Q$ cannot be volume-minimizing by the usual Stokes' Theorem argument for calibrated submanifolds.  

We now provide a few extra details, for which we need to describe $N$.  For maps $z_1,\ldots,z_n:\R\rightarrow\C$ define 
$$N=\{(t_1z_1(s),\ldots,t_n z_n(s))\in\C^n\,:\,s\in\R,\,t_1,\ldots,t_n\in\R,\, \sum_{j=1}^nt_j^2=1\}.$$
It is not difficult to calculate that $N$ is special Lagrangian with phase $i$ (so calibrated by $\Imm\Upsilon$) if and only if
$$\overline{z_j}\frac{\d z_j}{\d s}=i f_j \overline{z_1\ldots z_n}$$
for positive real functions $f_j$.  

Suppose that $f_j=1$ for all $j$.  Write $z_j=r_je^{i\theta_j}$, let $\theta=\sum_{j=1}^n\theta_j$ and suppose that $z_j(0)=c_j>0$.  From the differential equation, one quickly sees that $r_j^2=c_j^2+u$ for some function $u$ with $u(0)=0$ and $r_1\ldots r_n\cos\theta=c_1\ldots c_n$.  

If we now suppose that $u=t^2$, we see that 
$$\theta_j(t)=\int_0^{t}\frac{a_j\d t}{(1+a_jt^2)\sqrt{\frac{1}{t^2}\big((1+a_1t^2)\ldots(1+a_nt^2)-1\big)}}$$
where $a_j=c_j^{-2}$.  We observe that $\theta\rightarrow\pm\frac{\pi}{2}$ as $t\rightarrow\pm\infty$ and hence $N$, which is a Lawlor neck, is asymptotic to a pair of planes where the sum of the angles is $\pm\frac{\pi}{2}$.

Now fix $t>0$ and define 
$$f:X=\{(a_1,\ldots,a_n)\in\R^n\,:\,a_j\geq 0\}\rightarrow Y=\{(\theta_1,\ldots,\theta_n)\in\R^n\,:\,\theta_j\geq 0,\sum_{j=1}^n\theta_j<\frac{\pi}{2}\}$$
by $f(a_1,\ldots,a_n)=(\theta_1,\ldots,\theta_n)$ where
$$\theta_j=\int_0^{t}\frac{a_j\d t}{(1+a_jt^2)\sqrt{\frac{1}{t^2}\big((1+a_1t^2)\ldots(1+a_nt^2)-1\big)}}.$$
It is clear that if $n=1$, $f:X\rightarrow Y$ is surjective.  We want to show it is surjective for all $n$.
 For $\theta\in (0,\frac{\pi}{2})$ define $H_{\theta}=\{(\theta_1,\ldots,\theta_n)\in Y\,:\,\sum_{j=1}^n\theta_j=\theta\}$.  By our discussion above we see 
that $$f^{-1}(H_{\theta})\subseteq S_{\theta}=\{(a_1,\ldots,a_n)\in X\,:\,(1+a_1t^2)\ldots(1+a_nt^2)=\cos^{-2}t\}.$$
Notice that if the degree of $f:\partial S_{\theta}\rightarrow\partial H_{\theta}$ is 1 then the degree of 
$f:S_{\theta}\rightarrow H_{\theta}$ is 1.  Thus, by induction on $n$, we see  that $f:S_{\theta}\rightarrow H_{\theta}$ is surjective.  

Now, given any plane $\{(e^{i\theta_1}x_1,\ldots e^{i\theta_n}x_n)\,:\,(x_1,\ldots,x_n)\in\R^n\}$ where 
$(\theta_1,\ldots,\theta_n)\in\ Y$, $\theta_j\neq 0$ for all $j$, we see that we can choose a Lawlor neck $N$ which intersects the plane in a hypersurface as claimed.

We now move to the other direction in the statement of the Angle Theorem. If $P\cup Q$ does satisfy the angle condition, then (by choosing coordinates so that $P=\R^n$ and $Q$ is in standard position) we claim that it is calibrated by a so-called \emph{Nance calibration}:
$$\eta(u_1,\ldots,u_n)=\Ree\big((\d x_1+u_1\d y_1)\wedge\ldots\wedge(\d x_n+u_n\d y_n)\big)$$
where $u_1,\ldots,u_n\in\mathcal{S}^2\subseteq\Im\H$.  If $u_m=i$ for all $m$ then $\eta=\Ree\Upsilon$, so it is believable that it is a calibration in general, but we now show that it is indeed true.  

Let $x_1,y_1,\ldots,x_n,y_n$ be coordinates on $\R^{2n}$.  
We call an $n$-form $\eta$ on $\R^{2n}$ a torus form
if $\eta$ lies in the span of forms of type  $$\d x_{i_1}\wedge\ldots\wedge\d x_{i_k}\wedge \d y_{j_1}\wedge\ldots\wedge\d y_{j_l}$$ where $\{i_1,\ldots,i_k\}\cap\{j_1,\ldots,j_l\}=\emptyset$ 
and  $\{i_1,\ldots,i_k\}\cup\{j_1,\ldots,j_l\}=\{1,\ldots,n\}$.  We now claim that a torus form $\eta$ is a calibration if and only if 
$$\eta(\cos\theta_1e_1+\sin\theta_1e_{n+1},\ldots,\cos\theta_ne_n+\sin\theta_{n}e_{2n})\leq 1$$
for all $\theta_1,\ldots,\theta_n\in\R$.

For $n=1$, $\eta=\d x_1\wedge\d y_1$ which is a calibration.  Suppose that the result holds for $n=k$.  
Let $\eta$ be a torus form on $\R^{2(k+1)}$ and rescale $\eta$ so that the maximum of $\eta$ is $1$ and is attained at some plane.  
The idea is to show using the argument in the proof of Wirtinger's inequality to put planes in standard position that we can write 
$\eta=e_1\wedge \eta_1+e_2\wedge \eta_2$ where $e_1,e_2$ are orthonormal and span an $\R^2$ and $\eta_1,\eta_2$ are torus forms on $\R^{2k}$.  The claim then follows by induction on $n$.

Hence, the Nance calibration $\eta$ above is a calibration and moreover 
we know $P(\theta)$ is calibrated by $\eta(u)$ if and only if 
$$\prod_{j=1}^n(\cos\theta_j+\sin\theta_ju_j)=1.$$
We then just need to find the $u_j$ determined by $\theta_j$.  Notice that the condition that $\psi_1+\ldots+\psi_n\geq\pi$ holds if and only if
 $\theta_n\leq\theta_1+\ldots+\theta_{n-1}$.  
If we write $\cos\theta_j+\sin\theta_ju_j=w_j\overline{w}_{j+1}$ where $w_{n+1}=w_1$ and $w_j$ are unit imaginary quaternions then the product condition is
 satisfied and we just need 
$\langle w_j,w_{j+1}\rangle =\cos\theta_j$, which is equivalent to finding $n$ points on the unit 2-sphere so that $d(w_j,w_{j+1})=\theta_j$, where $\theta_n\leq
\theta_1+\ldots+\theta_{n-1}$. This is indeed possible, by considering an $n$-sided spherical polygon.
\end{proof}

\end{document}